\documentclass[12pt]{amsart}
\usepackage{amssymb,amsthm,amsmath,amstext}
\usepackage{mathrsfs}  
\usepackage{bm}        
\usepackage{mathtools} 
\usepackage{fullpage}
\usepackage{graphicx}
\usepackage[all]{xy}
\usepackage{float}
\usepackage{graphicx}
\usepackage{enumerate} 
\usepackage{amstext}
\usepackage{amsxtra}
\usepackage{tikz}
\usepackage{setspace}
\usepackage{listings}
\usepackage{booktabs}
\usepackage{biblatex}
\mathchardef\mhyphen="2D
\usepackage{hyperref}
\usepackage{breakurl}

\theoremstyle{plain}
\newtheorem{theorem}{Theorem}[section]
\newtheorem{prop}[theorem]{Proposition}
\newtheorem{lemma}[theorem]{Lemma}

\newtheorem{cor}[theorem]{Corollary}

\newcommand{\Brt}{{}_2\mathrm{Br}}

\newtheorem*{nonu-theorem}{Theorem}
\newtheorem{notation}[theorem]{Notation}

\theoremstyle{definition}

\theoremstyle{remark}

\newcommand{\sheaf}[1]{\mathscr{#1}}

\newcommand{\OO}{\mc{O}}

\renewcommand{\AA}{\sheaf{A}}

\newcommand{\PP}{\sheaf{P}}
\newcommand{\XX}{\sheaf{X}}
\newcommand{\YY}{\sheaf{Y}}

\newcommand{\mc}[1]{\mathcal #1}

\usepackage{hyperref}
\usepackage{color}
\hypersetup{colorlinks=true,linkcolor=blue,anchorcolor=blue,citecolor=blue,letterpaper=true}

\DeclareFontFamily{U}{wncy}{}
    \DeclareFontShape{U}{wncy}{m}{n}{<->wncyr10}{}
    \DeclareSymbolFont{mcy}{U}{wncy}{m}{n}
    \DeclareMathSymbol{\Sha}{\mathord}{mcy}{"58}

\begin{document}

\title[ Local-global principle]
{ Local-global principle for  hermitian  spaces over semi global fields}

\author{Jayanth Guhan} 

\begin{abstract}   Let $K$ be a  complete discrete valued field with residue field $k$
and $F$ the function field of a curve over $K$. Let $A \in {}_2Br(F)$ be  a central simple algebra with an involution 
$\sigma$ of any kind and $F_0 =F^{\sigma}$.  Let $h$ be an hermitian space over $(A, \sigma)$ and 
 $G = SU(A, \sigma, h)$ if $\sigma$ is of  first kind and $G = U(A, \sigma, h)$ if $\sigma$ is of second kind.
 Suppose that char$(k) \neq 2$ and ind$(A)\leq 4$.  Then we  prove that 
 projective  homogeneous spaces under $G$ over $F_0$ satisfy a local-global principle for rational points
with respect to discrete valuations of $F$. 
\end{abstract} 
 
\maketitle

Let $K$ be a complete discrete valued field and $F$ the function field of a curve over $K$.
Let  $\Omega$ be  a set of discrete valuations of $F$.  For $\nu \in F$, let $F_\nu$ denote the 
completion of $F$ at $\nu$. Let $G$ be a connected linear algebraic group over $F$ and 
 $X$ a projective homogeneous variety  under $G$   over $F$.
We say that a {\it local-global principle}  holds for $X$ with respect to $\Omega$
if $X(F_\nu) \neq \emptyset$ for all $\nu \in \Omega$ implies $X(F) \neq \emptyset$. 

If $K$ is a $p$-adic field, $F$ the function field of a curve over $K$
and $G$ a connected linear algebraic group which of classical type with characteristic of the residue field $k$ a
'good' prime for $G$, then local-global principle holds for projective homogeneous spaces under $G$ over $F$ with 
respect to  the set $\Omega_F$ of all 'divisorial' discrete valuations (\cite{patchingLGPpadic}, \cite{reddy}, \cite{wu},
 \cite{parimala2020localglobal}).
The aim of this paper is  to  obtain analogues  results for certain unitary groups  if $K$ is a complete discrete 
valued field with any residue field. 

Suppose $K$ is any complete discrete valued field with residue field $k$. 
Let $q$ be a quadratic form over $F$ of rank at least 3 and $G = SO(q)$. 
If char$(k) \neq 2$, 
then Colliot-Th\'el\`ene-Parimala-Suresh  proved that local-global principle holds for 
projective homogeneous spaces under $G$ over $F$ with respect to $\Omega_F$ (\cite[Theorem 3.1]{patchingLGPpadic}). 
Let $A$ be a central simple algebra over $F$ of degree $n$ and $G = PGL_1(A)$. 
If $n$ is coprime to char$(k)$ and $k$ contains a primitive $n^{\rm th}$ root of unity,
then Reddy-Suresh proved that local-global principle holds for 
projective homogeneous spaces under $G$ over $F$ with respect to $\Omega_F$ (\cite[Theorem 2.6]{reddy}).
 
 Let $A \in {}_2Br(F)$ be a central simple algebra over $F$ with an involution $\sigma$ and $h$ a  hermitian form 
 over $(A, \sigma)$ and $G = SU(A, \sigma, h)$ if $\sigma$ is of first kind and $G = U(A, \sigma, h)$ if $\sigma$ is of second kind.
 Suppose that char$(k) \neq 2$. 
 A theorem of Wu (\cite[Theorem 1.2]{wu}) asserts that if   ind$(A) \leq 2$ or for every finite extension 
 $\ell/k$, $\mid \ell^*/\ell^{*2} \mid \leq 2$, 
 then local-global principle holds for 
projective homogeneous spaces under $G$ over $F$ with respect to $\Omega_F$. 
Since for every finite field ${\mathbb F}$ of characteristic not 2,  $\mid {\mathbb F}^*/{\mathbb F}^{*2} \mid = 2$, 
  if $K$ is a $p$-adic field with $p \neq 2$, then  local-global principle holds for 
projective homogeneous spaces under $G$ over $F$ with respect $\Omega_F$.

These results are in the direction of a Hasse principle for projective homogeneous spaces over general semiglobal fields. 
We however point out that a Hasse principle for principal  or projective homogeneous spaces over function fields of $p$-adic curves is 
a conjecture in (\cite{patchingLGPpadic}) and proved in many cases (cf.  \cite{parimala2020localglobal}).
There are  examples  to show  that a Hasse principle may fail for principal homogeneous space for arbitrary semiglobal field
(\cite{colliotthélène2021localglobal}, \cite{tori}).

In this paper we prove the following (cf. \ref{main-theorem1}, \ref{main-theorem2}).

 \begin{theorem}  
  Let $K$ be a complete discretely valued field with valuation ring $T$ and residue field $k$.
 Suppose that char$(k) \neq 2$. Let $F $ be the function field of a smooth projective geometrically 
 integral curve over $K$.   
 Let $A \in {}_2Br(F)$ be a central simple algebra over $F$ with an involution $\sigma$ of any kind, $F_0 = F^\sigma$
  and $h$ a hermitian form over $(A, \sigma)$. 
   Let $G = SU(A, \sigma, h)$ if $\sigma$ is first kind or 
 $U(A,  \sigma, h)$ if $\sigma$ is of second kind.  Suppose that one of the following holds:\\
 i) ind$(A) \leq 4$ \\
 ii) for every finite extension $\ell/k$,   every element  $D \in {}_2Br(\ell)$ has index at most 2. 
  Let $X$ be a projective homogeneous variety under $G$ over $F_0$.
 If $X(F_{0\nu}) \neq \emptyset$ for all divisorial  discrete valuations $\nu$ of $F $, then $X(F_0 ) \neq \emptyset$.  
 \end{theorem}
 
 As a consequence we get the following (\ref{springer}). 
 
 \begin{cor}
 Let $K$ be a complete discretely valued field with valuation ring $T$ and residue field $k$.
 Suppose that  $k$ is a local field   with  the characteristic of the  residue field   not equal  2. 
 Let $F $ be the function field of a smooth projective geometrically 
 integral curve over $K$.   
 Let $A \in {}_2Br(F)$ be a central simple algebra over $F$ with an   involution $\sigma$   and
 $h$ a hermitian form over $(A, \sigma)$ and $F_0 =  F^\sigma$.   
If $h \otimes L$ is isotropic for some odd degree extension $L/F_0$, then $h$ is isotropic. 
 \end{cor} 

The main technique in   the proof of the main theorem  is the patching technique of Harbater, Hartman and Krashen 
(\cite[Theorem 3.7]{HHK-quadratic}) and the existence of maximal orders over two dimensional complete regular local rings  for some suitable division algebras. 
We now give a brief description of  the structure of paper. 
In  \S(\ref{prel}),  we recall the patching set up and some known results. 
Let $R$ be a 2-dimensional complete regular local ring  with residue field $k$, maximal ideal $(\pi, \delta)$
 and field of fractions $F$.
Suppose that char$(k) \neq 2$. Let $D $  be a division algebra over $F$ which is unramified on $R$ except possibly at 
$(\pi)$ and $(\delta)$. 
If $D$ has an involution of any kind, 
in \S \ref{disc} and  \S\ref{disc2}, following Saltman (\cite[Theorem 2.1]{saltman-cyclic} \& \cite[Theorem 1.2]{saltman-divison}), we give a description of $D$ if ind$(D) \leq 4$ or  
 if for every finite extension $\ell$ of $k$, every element of $_2Br(\ell)$ has index at most 2.
 In \S \ref{maximal-order}, we produce some suitable maximal $R$-orders for such algebras. 
 In \S \ref{lgp-local}, under some assumptions on the existence of maximal $R$-orders in $D$, we show that 
 hermitian spaces satisfy a certain local-global principle.  In \S\ref{blowups}, we study the behavior of the maximal $R$-orders
 constructed in \S \ref{maximal-order} under blowups.   Finally using all these we prove the main results in \S\ref{main}.

\section{Preliminaries}
\label{prel} 

In this section we recall some basic facts about hermitian forms (\cite[Chapter 7 \& 8]{scharlau}), 
projective homogeneous spaces  under  unitary groups  (\cite{merkurjevflagvar}, \cite{merkurjevflagvar2}) and patching techniques (\cite{HHK-quadratic}). 

Let $F$ be a field and $A$  a central simple algebra with an involution $\sigma$.
Let $F_0 = F^{\sigma} = \{ a \in F\mid \sigma(a) = a \}$. Then $[F : F_0] \leq 2$.
If $F = F_0$, then $\sigma$ is  called an involution of the \textit{first kind} if $[F : F_0] = 2$, then 
$\sigma$ is   called an involution of the \textit{second kind}. 
We  also say that the involution $\sigma$ is an $F/F_0$-involution.  
Let $h$ be an hermitian form over $(A, \sigma)$.
Let $G(A, \sigma, h) = SU(A, \sigma, h)$ if $F = F_0$ and $G(A, \sigma, h) = U(A, \sigma, h)$ if $[F : F_0] = 2$.
Then $G(A, \sigma, h)$ is a connected  linear algebraic group over $F_0$ which is rational (\cite[Lemma 5]{chernusov}). 

By a theorem of Wedderburn (\cite[Chapter 8, Corollary 1.6 \& Theorem 1.9]{scharlau}), we have $A \simeq M_n(D)$ for a central division algebra over $D$ over $F$.
 Since  $A$ has a  $F/F_0$-involution, $D$ has a  $F/F_0$-involution (\cite[Chapter 8, Corollary 8.3]{scharlau}). 
 Let $\sigma_0$ be a $F/F_0$-involution on $D$. Then, by Morita equivalence (\cite[Chapter 1, 9.3.5]{knus-hermitian}), $h$ corresponds to a hermitian 
 $h_0$ over $(D, \sigma_0)$.

Let $X$ be a projective homogeneous space under $G(A, \sigma, h)$ over $F_0$.
Then $X$ corresponds to a projective homogeneous space $X_0$  under $G(D, \sigma_0, h_0)$ over $F_0$ such that 
for any extension $L/F_0$,  $X(L) \neq \emptyset$ if and only if $X_0(L) \neq \emptyset$ (\cite[Section 5]{merkurjevflagvar}, \cite[Section 9]{merkurjevflagvar2}, \cite[Theorem 16.10]{karpenko}). 
Hence the study of rational points on   projective homogeneous space under $G(A, \sigma, h)$  reduces to 
the study of rational points on   projective homogeneous space under $G(D, \sigma_0, h_0)$.  
Suppose   $\sigma'$ is an involution on  $A$ of the same kind as $\sigma$. Then $\sigma' = $ int$(u) \sigma$ for some 
$u \in A^*$ and $h' = uh$ is an hermitian form over $(A, \sigma')$. Further  $X$ corresponds to a 
projective homogeneous space $X'$ under $G(A,\sigma', h')$  over $F_0$ such that 
for any extension $L/F_0$,  $X(L) \neq \emptyset$ if and only if $X'(L) \neq \emptyset$. 
For more details on the structure of projective homogeneous space under $G(A, \sigma, h)$, we refer to 
(\cite{merkurjevflagvar}, \cite{merkurjevflagvar2}).

Let $\XX$ be a regular integral scheme with function field $F$. 
For a point $P$ of $\XX$, let  $\OO_P$ denote the local ring at $P$ on $\XX$, 
$\hat{\OO}_P$ the completion of the local ring $\OO_P$.  
Let $A$ be a central simple algebra over $F$.  Let $P \in \XX$.
We say that $A$ is {\it unramified}  at $P$, if there exists an Azumaya algebra  $\AA$ 
over the local ring $\OO_P$ at $P$ such that $\AA \otimes F \simeq A$. 
If $A$ is not unramified at $P$, we say that $A$ is {\it ramified } at $P$.

Let $T$ be a  complete discrete  valuation ring   with residue field $k$ and field of fractions $K$. 
Let $F$ be the function field of a smooth projective  geometrically integral   curve over $K$.
A two dimensional projective scheme $\XX \to $ Spec$(T)$ with function field $F$ is  called a {\it model} of $F$. 
The fibre $X_0$ over the closed point of Spec$(T)$ is called the {\it closed fibre} of $\XX$. 

Let $\XX$  be  a  model  of $F$ with $X_0$ its closed fibre.  
Every codimension one point $x$ of $\XX$ gives a discrete valuation  $\nu_x$ on $F$.
A discrete   valuation  of $F$ is called a {\it divisorial discrete valuation} of $F$, if it is given by a codimension 
one point of a model of  $F$.   For any discrete valuation $\nu$ of $F$, we denote the 
completion of $F$ at $\nu$ by $F_\nu$. 

Let $A$ be a central simple algebra over $F$.
The {\it ramification divisor}  ram$_\XX(A)$ of $A$ on $\XX$ is
 the union of codimension one points of $\XX$ where $A$ is ramified. 
 Note that there are only finitely many codimension one point of $\XX$ where $A$ ramified.  

For any point $P \in \XX$, let $\hat{\OO}_P$ be the   completion of the local ring 
$\OO_P$ at $P$ on $\XX$ and $F_P$ the field of fractions of $\hat{\OO}_P$. 
Let $U \subset X_0$ be a nonempty open subset of the nonsingular points 
of  $X_0$ which is properly contained in an irreducible component of 
$X_0$. Let $R_U$ be the set of $f \in F$ which are regular at every point of $U$. Then $T \subset R_U$. 
Let $t \in T$ be a parameter and  $\hat{R}_U$ be  the $(t)$-adic completion of $R_U$.
Let $F_U$ be the field of fractions of $\hat{R}_U$.  For any nonempty subset $V$ of $U$, we have 
$F_U \subset F_V$.
For any point $P \in U$, we have $F_U \subset  F_P$. For more details about $F_U$ and $F_P$, we refer to \cite{HHK-quadratic}.

We record the following theorem  from
 (\cite[Proposition 5.8]{HHK-torsors}).

\begin{theorem} 
\label{finite} Let $T$ be a  complete discrete  valuation ring   with residue field $k$ and field of fractions $K$. 
Let $F$ be the function field of a smooth projective  geometrically integral   curve over $K$ and  $\XX  $  a  model  of $F$
with $X_0$ its closed fibre.  Let $Y$ be a variety over $F$. Suppose that $Y(F_\nu) \neq \emptyset$ for all 
divisorial discrete valuations $\nu$ of $F$.  For every irreducible component $C$ of $X_0$, there exists a nonempty 
proper open subset $U$ of $C$ such that $Y(F_{U}) \neq \emptyset$. In particular there exists a finite subset 
$\PP$ of  closed points of $X_0$ such that $Y(F_{P}) \neq \emptyset$ for all $P \in X_0 \setminus \PP$. 
\end{theorem}  

\begin{proof} Let  $C$ be an  irreducible component  of $X_0$. Then  for the generic point $\eta$ of $C$, $F_\eta$ is 
the completion of  $F$ at the discrete valuation given by $\eta$ and hence $Y(F_\eta) \neq \emptyset$. 
Then, by  (\cite[Proposition 5.8]{HHK-torsors}), 
there exists a nonempty 
proper open subset $U_\eta$ of $C$ such that $Y(F_{U_\eta}) \neq \emptyset$. Let $\PP = X_0 \setminus \cup_\eta U_\eta$.
Since $U_\eta$ is a nonempty open subset of the irreducible curve $C$, $C - U_\eta$ is a finite set.
Since there are only finitely many  irreducible components  of $X_0$, $\PP$ is a finite set.
Let $P \in X_0 \setminus \PP$. Then $P \in U_\eta$ for some $\eta$. 
Since $F_{U_\eta} \subset F_P$ and $Y(F_{U_\eta}) \neq \emptyset$, $Y(F_{P}) \neq \emptyset$. 
\end{proof}

We record the following from (\cite[Proposition 2.4]{reddy}).

\begin{prop}
\label{reddy}
Let $(R,\mathfrak{m})$ be a 2-dimensional complete regular ring with maximal ideal $\mathfrak{m} = (\pi,\delta)$,
  field of fractions $F$ and residue field $k$.  Let $D \in {}_d\text{Br}(F)$ be  a division algebra over $F$ which is  
  unramified  on $R$   except  possibly at $\langle \pi \rangle$ or $\langle \delta \rangle$. Suppose that  $d$ is coprime to 
   $\text{char}(k)$ and $F$ contains a primitive $d^{\rm th}$ root of unity. Then 
   ind$(A) = $ ind$(A\otimes F_\pi)$.
\end{prop} 

\begin{proof} Let $n = $ deg$(A)$.  This result is  proved in (\cite[Proposition 2.4]{reddy}) under the assumption that 
$F$ contains a primitive $n^{\rm th}$ root of unity. However we note that the same proof goes through if $F$ contains a primitive 
$d^{\rm th}$ root of unity. 
\end{proof}

We end this section with the following.  

\begin{theorem} 
\label{pgla} Let $T$ be a  complete discrete  valuation ring   with residue field $k$ and field of fractions $K$. 
Suppose that char$(k) \neq 2$. 
Let $F$ be the function field of a smooth projective  geometrically integral   curve over $K$ and  $\XX  $  a  model  of $F$. 
Let $A$ be a central simple algebra over $F$ of period  at most 2.  Let  $L$ be the field $F$ or $F_P$ for some closed point $P$ of $\XX$. 
Let $Z$ be a projective homogeneous space under $PGL(A)$ over $L$.  If $Z(L_\nu) \neq \emptyset$ for all divisorial discrete valuations of $L$,
then $Z(L) \neq \emptyset$.
\end{theorem} 

\begin{proof}  This follows from the proof of (\cite[Theorem 10.1]{parimala2020localglobal}) and 
 (\ref{reddy}).
  \end{proof} 
 
 \begin{cor} 
 \label{gla}
 Let $L$ and $A$ be as in (\ref{pgla}). Let $Z$ be a projective homogenous space under $GL(A)$ over $L$.
  If $Z(L_\nu) \neq \emptyset$ for all divisorial discrete valuations of $L$,
then $Z(L) \neq \emptyset$.
 \end{cor}
 
 \begin{proof}  This follows from the proof of (\cite[Corollary 10.3]{parimala2020localglobal}) and  (\ref{pgla}).
  \end{proof}

\section{Division algebras  with involution of first kind over two dimensional local fields }
\label{disc}

 Let $(R,\mathfrak{m})$ be a 2-dimensional complete regular ring with maximal ideal $(\pi,\delta)$,
  field of fractions $F$ and residue field $k$. Suppose that  $\text{char}(k) \neq 2$.  Let $D \in {}_2\text{Br}(F)$ be  a division algebra over $F$ which is  
  unramified  on $R$   except  possibly at $\langle \pi \rangle$ or $\langle \delta \rangle$.  In this section we show that  if ind$(D) = 4$, then 
  $D$ is   a tensor product of  two quaternion  algebras with some properties.   
  Suppose  that for any central simple algebra  $A\in  {}_2\text{Br}(k)$,   ind$(A) \leq 2$. Then we show that ind$(D) \leq 8$. 
  Further we show that  if  ind$(D) = 8$, then $D$  is isomorphic to a  tensor product of three quaternion algebras with some properties.

We begin with the following well know lemma

\begin{lemma}
\label{overk} Let $A \in {}_2\text{Br}(k)$ be a  central  division  algebra over  $k$. \\
1) If $ A \otimes_k k(\sqrt{a})$ has index at most 2 for some $a \in k^*$, then $A = (a, c) \otimes (b, d) \in Br(k)$ for some $b, c, d, \in k^*$. \\
2) If $ A \otimes_k k(\sqrt{a}, \sqrt{b})$ is a matrix algebra for some $a, b \in k^*$, then $A = (a,c) \otimes (b, d) \in Br(k)$ for some $c, d, \in k^*$. 
\end{lemma}

\begin{proof}    If deg$(A) = 1$ or $2$, then 1) and 2) are immediate.    Suppose that deg$(A) \geq 4$.

Suppose $ A \otimes_k k(\sqrt{a})$ has index at most 2 for some $a \in k^*$.  Then  deg$(A) = $ ind$(A) = 4$ and $ K = k(\sqrt{a})$ is a subfield of $A$. 
Let $A_1$ be the commutant of $ K$ in $A$. Then $A_1$ is a quaternion algebra over $K$ (\cite[Theorem 5.4]{scharlau}). Since $A \in   {}_2\text{Br}(k)$, $A$ admits an involution (\cite[Chapter 8, Theorem 8.4]{scharlau})
and the non trivial automorphism of $K/k$ can be extended to an involution $\sigma$ on $A$ (\cite[Chapter 8, Theorem 10.1]{scharlau}). Since $\sigma(K)= K$,  the restriction of $\sigma$ to $A_1$
is an   involution of second kind. Thus, by a theorem of Albert (\cite[Chapter 10, Theorem 21]{albert}),  $A_1 = K \otimes Q_1$ for some quaternion algebra $Q_1$. 
Let $Q_2$ be the commutant  of $Q_1$ in $A$. Then $K \subset Q_2$ and $A = Q_2 \otimes Q_1$. 
Since $K \subset Q_2$, $Q_2 = (a, c)$ for some $c \in k^*$. Since $Q_1$ is a quaternion algebra, $Q_1 = (b, d)$ for some $b, d \in k^*$.
Hence $A = (a, c) \otimes (b, d)$.

Suppose $ A \otimes_k k(\sqrt{a}, \sqrt{b})$ is a matrix algebra for some $a, b \in k^*$.  Then deg$(A) = $ ind$(A) = 4$ and $ K = k(\sqrt{a}, \sqrt{b} )$ is a 
maximal subfield of $A$. 
Then as above we have $A = Q_2 \otimes Q_1$ with $K = k(\sqrt{a}) \subset Q_2$. Since $k(\sqrt{a}, \sqrt{b})$ is a maximal subfield of $A$, it follows that 
$k(\sqrt{b}) \subset Q_1$. Hence $Q_2 = (a, c)$ and $Q_1 = (b, d)$ for some $c, d \in k^*$. Thus $A = (a,c) \otimes (b, d)$. 
\end{proof}

\begin{lemma}
\label{deg4-u-field}
 Let $R$ be a  complete regular local ring with residue field $k$ and field of fractions $F$.
Suppose that char$(k) \neq 2$. Let $\Gamma_0$ be an Azumaya algebra over $R$ and 
$D_0 = \Gamma_0 \otimes_RF \in {}_2\text{Br}(F)$.  Let $u \in R$ be a unit. 
If   ind$(D_0\otimes_F(F(\sqrt{u}))) \leq 2$.
Then there exists $v, w, t \in R^*$ such that $D_0 =  (u, v) \otimes (w, t) \in Br(F)$.
\end{lemma}
\begin{proof}
Suppose that ind$(D_0\otimes_F(F(\sqrt{u}))) \leq 2$. Let $\mathfrak{D}_0 = \Gamma_0 \otimes_R k$.   Since ind$(D_0 \otimes F(\sqrt{u}) \leq 2$, 
  $\mathfrak{D}_0 \otimes k(\sqrt{\overline{u}}) \leq 2$ (\cite[Lemma 1.1]{reddy}), where $\overline{u}$ is the image of $u$ in $k$. 
   Thus there exist $b, c, d \in k^*$ such that $\mathfrak{D}_0 = (\overline{u}, c) \otimes (b, d)$ (cf. \ref{overk}). 
   Let $v,w,t \in R^*$ be the lifts of $ b, c, d \in k^*$. 
   Since $R$ is a complete regular  local ring,  $\text{Br}(R) \cong \text{Br}(k)$ (\cite{auslander}, 6.5).  Hence  $D_0 = (u, v) \otimes (w, t) \in Br(F)$. 
\end{proof}
 
\begin{lemma}
\label{deg4-uv-field}
 Let $R$ be a  complete regular local ring with residue field $k$ and field of fractions $F$.
Suppose that char$(k) \neq 2$. 
Let $\Gamma_0 $ be an Azumaya algebra over $R$  with 
$D_0 = \Gamma_0 \otimes_RF \in {}_2\text{Br}(F)$.  
Let $u, v \in R$ be units. 
If   $D_0\otimes_F(F(\sqrt{u}, \sqrt{v})) $ is a matrix algebra, 
then there exists $ w, t \in R$ units such that $D_0 = (u, w) \otimes (v, t) \in \text{Br}(F)$.
\end{lemma}
\begin{proof} Let $\mathfrak{D}_0 = \Gamma_0 \otimes_R k$.
Since  $D_0\otimes_F(F(\sqrt{u}, \sqrt{v})) $ is a matrix algebra, $\mathfrak{D}_0 \otimes k(\sqrt{\overline{u}},  \sqrt{\overline{v}}) $ is a matrix algebra.
Hence $\mathfrak{D}_0  = (\overline{u}, c) \otimes ( \overline{v}, d) \in  Br(k) $ (cf. \ref{overk}). 
Let $w,t \in R^*$ be lifts of $c, d \in k^*$.  Since $R$ is a complete regular  local ring, $D_0  = (u, w) \otimes (v, t) \in \text{Br}(F)$.  
\end{proof}

\begin{lemma}  
\label{index-cal}
Let $R$ be a two dimensional 
 complete regular local ring  with residue field $k$, maximal ideal 
 $(\pi, \delta)$ and field of fractions $F$.
 Suppose that char$(k) \neq 2$.  
 Let $\Gamma_0$ be an Azumaya algebra on $R$  and $D_0 = \Gamma_0 \otimes F$.
 Suppose that  $D_0 \in \Brt(F)$.
 Let   $u, v \in R$ units. \\
 i) If $D =  D_0   \otimes (u, \pi)  \otimes (v, \delta)$, then ind$(D) = $ ind$(D_0\otimes F(\sqrt{u}, \sqrt{v})) [ F(\sqrt{u}, \sqrt{v}): F]$. \\
 ii) If $D = D_0 \otimes (u\pi, v\delta)$, then ind$(D) = $ 2 ind$(D_0)$. 
 \end{lemma}
 
 \begin{proof}
 Let $\kappa(\pi)$ be the residue field at $\pi$.  Then $\kappa(\pi)$ is the field of fractions of 
the complete discrete valuation ring $R/(\pi)$ and the  the image $\overline{\delta}$  of $\delta$ in $\kappa(\pi)$ is a parameter. 
 
Suppose $D =  D_0   \otimes (u, \pi)  \otimes (v, \delta)$.  Since $D_0 \otimes (v, \delta)$ is unramified at $\pi$, 
 by (\cite[Lemma 2.1]{reddy}), we have ind$(D\otimes F_\pi) = $ ind$((D_0 \otimes (v, \delta) \otimes F_\pi(\sqrt{u}))  [F_\pi(\sqrt{u}] : F_\pi]$.
 Since $\Gamma_0 \otimes \kappa(\pi)$ is unramified on $R/(\delta)$, by (\cite[Lemma 2.1]{reddy}), we have 
 ind$( \Gamma_0    \otimes (\overline{v}, \overline{\delta}) \otimes \kappa(\pi)(\sqrt{\overline{u}}))   = $ 
 ind$(\Gamma_0 \otimes \kappa(\pi)(\sqrt{\overline{u}}, \sqrt{\overline{v}}))[\kappa(\pi)(\sqrt{\overline{u}}, \sqrt{\overline{v}}) : \kappa(\pi)(\sqrt{\overline{u}})]$. 
 Since $F_\pi$ is complete, we have 
 $${\text ind}((D_0 \otimes (v, \delta) \otimes F_\pi (\sqrt{u})) =  
 {\text ind}(D_0 \otimes F_{\pi}(\sqrt{u}, \sqrt{v})) [F_{\pi}(\sqrt{u}, \sqrt{v}) :  F_{\pi}(\sqrt{u})].$$
 Hence ind$(D\otimes F_\pi) = $ ind$((D_0 \otimes  F_\pi(\sqrt{u}, \sqrt{v}))  [F_\pi(\sqrt{u}, \sqrt{v}] : F_\pi]$.
 By (\ref{reddy}), we have ind$(D) = $ ind$(D\otimes F_\pi)$, ind$(D_0 \otimes F_\pi(\sqrt{u}, \sqrt{v})) = $ ind$((D_0 \otimes  F_\pi(\sqrt{u}, \sqrt{v})))$
 and $[F_\pi(\sqrt{u}, \sqrt{v}): F_\pi] = [F(\sqrt{u}, \sqrt{v}): F]$.
 Thus ind$(D) = $ ind$(D_0\otimes F(\sqrt{u}, \sqrt{v})) [ F(\sqrt{u}, \sqrt{v}): F]$.
 
 Suppose $D = D_0 \otimes (u\pi, v\delta)$. Then as above, we have ind$(D) =$ ind$(D_0 \otimes F(\sqrt{v\delta})) [F(\sqrt{v\delta}) : F]$.
 Since $D_0$ is unramified at $\delta$, we have ind$(D_0  \otimes F_\delta(\sqrt{v\delta})) = $ ind$(D_0  \otimes F_\delta)$. 
 Once again, by (\cite{reddy}), we have ind$(D_0  \otimes F(\sqrt{v\delta})) = $  ind$(D_0  \otimes F_\delta(\sqrt{v\delta})) = $ ind$(D_0  \otimes F_\delta)$ = 
 ind$(D_0 )$.  Since $[F(\sqrt{v\delta}) : F] = 2$, we have ind$(D) = $ 2 ind$(D_0)$.  
  \end{proof}
  
  We recall the following. 
 
\begin{lemma}
\label{wu-disc}
 (\cite[Lemma 3.6]{wu}) 
Let $R$ be a two dimensional 
 complete regular local ring  with residue field $k$, maximal ideal 
 $(\pi, \delta)$ and field of fractions $F$.
 Suppose that char$(k) \neq 2$. 
  Let $D $ be quaternion  division algebra over $F$ 
which is unramified on $R$ except possibly at $(\pi)$ and $(\delta)$. 
 Then $D$ is isomorphic to  one of the following: \\
i) $(u, w) $ \\
ii) $  (u, v\pi)$ or $ (u, v\delta)$ \\
iii)  $ (u, v\pi\delta)$\\
 iv)  $ (u\pi, v\delta)$ \\
for some units $u, v,w,  \in R$. 
\end{lemma}

\begin{prop}
\label{disc-2dim}
Let $R$ be a two dimensional 
 complete regular local ring  with residue field $k$, maximal ideal 
 $(\pi, \delta)$ and field of fractions $F$.
 Suppose that char$(k) \neq 2$. 
  Let $D \in {}_2\text{Br}(F)$ be a division algebra over $F$ 
which is unramified on $R$ except possibly at $(\pi)$ and $(\delta)$. 
Suppose that ind$(D)  =  4$.   
Then $D$ is isomorphic to  one of the following: \\
i) $(u, w) \otimes (v, t)$ \\
ii) $(u, w) \otimes (v, t\pi)$ or $(u, w) \otimes (v, t\delta)$ \\
iii)  $(u, v) \otimes (w, t\pi\delta)$\\
iv) $(u, w\pi) \otimes (v, t\delta)$\\ 
v)  $(u, v) \otimes (w\pi, t\delta)$ \\
for some units $u, v,w, t \in R$. 
\end{prop}

\begin{proof}
Suppose that $D$ is unramified on $R$.   Let $\Gamma$ be an Azumaya algebra on $R$ with $\Gamma \otimes F \simeq D$ (\cite[Theorem 7.4]{auslander}).
Since ind$(D) = 4$, we have ind$(\Gamma \otimes k) = 4$. Hence $\Gamma \otimes k = (a, b) \otimes (c, d)$ for some $a, b, c, d \in k^*$.
Let $u, w, v, t \in R$ be lifts of $a, b, c, d$. Since $R$ is complete, we have $\Gamma = (u, w) \otimes (v, t) \in \text{Br}(R)$ and 
hence $D = (u, w) \otimes (v, t) \in \text{Br}(F)$.  Since  deg$(D) = 4$, $D \simeq (u, w) \otimes (v, t)$.

 Suppose that $D$ is ramified only at $\langle \pi \rangle$. 
 Then, by Saltman's classification (\cite{wu}, Proposition 3.5), 
  we have $D = D_0 \otimes  (u,\pi)$, where $D_0$ is unramified on $R$ and $u \in R$ a unit which is not a square. 
  Since $D = D_0 \otimes (u, \pi) \otimes (1, \delta)$, 
  by (\ref{index-cal}(i)), we have ind$(D) = 2 $ ind$(D_0 \otimes F(\sqrt{u}))$. 
  Since ind$(D) = 4$,  we have ind$(D_0 \otimes F(\sqrt{u}) ) = 2$.  
 Hence, by (\ref{deg4-u-field}), we have $D_0 = (u, w) \otimes (v, t)$  for some $u, v, w, t \in R$ units. 
 We have $D = D_0 \otimes (u, \pi) = (u, w) \otimes (v, t) \otimes (u, \pi) = (u, w\pi) \otimes (v, t)$.  
 Since ind$(D) = 4$, $D \simeq (u, w\pi) \otimes (v, t)$. 
 Similarly if  $D$ is ramified only at $\langle \delta \rangle$, then $D \simeq (u, w\delta) \otimes (v, t)$.

Suppose that $D$ is ramified both at $\langle \pi \rangle$ and $\langle \delta \rangle$.
Then by (\cite[Theorem 2.1]{saltman-cyclic} \& \cite[Theorem 1.2]{saltman-divison}), we have $ D = D_0 \otimes (u, \pi) \otimes (v,\delta)$ or $D = D_0  \otimes (w\pi, t \delta)$
 for some $u, v \in R - R^2$ units, $w, t\in R$ units  and $D_0$ unramified on $R$.

Suppose   $D = D_0  \otimes (u, \pi)  \otimes (v,\delta) $.  Suppose that $uv \in R$ is a square.
Then $D = D_0  \otimes (u, \pi \delta)$.   By (\ref{index-cal}(i)), we have  ind$(D) = $  2 ind$(D_0\otimes F(\sqrt{u}))$  and hence 
ind$(D_0\otimes F(\sqrt{u})) = 2$. Thus, by  (\ref{deg4-u-field}), we have $D_0 = (u, w) \otimes (v, t)$  for some $u, v, w, t \in R$ units. 
In particular  $D \simeq (v, t ) \otimes (u,  w\pi\delta)$. 

Suppose that $uv$ not a square in $R$. By (\ref{index-cal}(i)), we have  ind$(D) = $  4 ind$(D_0\otimes F(\sqrt{u}))$  and hence 
ind$(D_0\otimes F(\sqrt{u})) = 1$. 
Hence $D_0\otimes F(\sqrt{u}, \sqrt{v})$ is a matrix algebra.
Thus, by (\ref{deg4-uv-field}), we have $D_0 = (u, w) \otimes (v, t)$ for some units $w, t \in R$.
In particular $D \simeq (u, w\pi) \otimes (v, t\delta)$.

Suppose $D = D_0 + (w\pi, t \delta)$.    By (\ref{index-cal}(ii)), we have  ind$(D) = $  2 ind$(D_0)$
and hence ind$(D_0) = 2$. Thus $D_0 = (u, v)$ and $D \simeq (u, v) \otimes (w\pi, t\delta)$. 
\end{proof}

\begin{prop}
\label{deg8}
Let $R$ be a two dimensional 
 complete regular local ring  with residue field $k$, maximal ideal 
 $(\pi, \delta)$ and field of fractions $F$.
 Suppose that char$(k) \neq 2$ and every central simple algebra  in ${}_2\text{Br}(k)$ has index at most 2. 
  Let $D \in {}_2\text{Br}(F)$ be a division algebra over $F$ 
which is unramified on $R$ except possibly at $(\pi)$ and $(\delta)$.  Then ind$(D) \leq 8$.
Further if ind$(D) = 8$, then $D \simeq (w, t) \otimes (u, \pi) \otimes (v, \delta)$ for some units $u, v, w, t \in R$. 
\end{prop}

\begin{proof}
Suppose that $D$ is unramified on $R$.   
 Let $\Gamma$ be an Azumaya algebra on $R$ with $\Gamma \otimes F \simeq D$ (\cite[Theorem 7.4]{auslander}).
 By the assumption on $k$,  ind$(\Gamma \otimes k) \leq 2$.   Since $R$ is complete, we have  ind$(D) \leq 2$ (\cite[Theorem 6.5]{auslander}).

Suppose that $D$ is ramified only at $\langle \pi \rangle$. 
 Then, by Saltman's classification (\cite{wu}, Proposition 3.5), 
  we have $D = D_0 \otimes  (u,\pi)$, where $D_0$ is unramified on $R$ and $u \in R$ a unit which is not a square. 
  Since $D_0$ is unramified on $R$, by the assumption on $k$,  ind$(D_0)  \leq 2$ and hence ind$(D) \leq 4$. 
  Similarly  if $D$ is ramified only at $\langle \delta \rangle$, then ind$(D) \leq 4$.

Suppose that $D$ is ramified both at $\langle \pi \rangle$ and $\langle \delta \rangle$.
Then by (\cite[Theorem 2.1]{saltman-cyclic} \& \cite[Theorem 1.2]{saltman-divison}), we have $ D = D_0 \otimes (u, \pi) \otimes (v,\delta)$ or $D = D_0  \otimes (w\pi, t \delta)$
 for some $u, v \in R - R^2$ units, $w, t\in R$ units  and $D_0$ unramified on $R$. 
 Since  $D_0$ is unramified on $R$, by the assumption on $k$,  ind$(D_0)  \leq 2$.
In particular ind$(D) \leq 8$.

Suppose ind$(D) =8$.  Then $D =  D_0 \otimes (u, \pi) \otimes (v,\delta)$ for some units  $u, v \in R $ units   and $D_0$ unramified on $R$.
Once again by the assumption on $k$, we have $D_0 = (w, t)$ for some units $w,  t \in R$.
Hence $D \simeq  (w, t)  \otimes (u, \pi) \otimes (v,\delta)$. 
\end{proof}

\section{Two torsion division algebras  with involution of second  kind over two dimensional local fields }
\label{disc2}

Let $R_0$ be a 2-dimensional complete regular local ring with maximal ideal $\mathfrak{m}_0 = (\pi_0,\delta_0)$ and residue field $k_0$.
Suppose that char$(k_0) \neq 2$. 
Let $F_0$ be the field of fractions of $R_0$ and let $F = F_0(\sqrt{\lambda})$ be an extension of degree 2, with $\lambda $ a unit in $R_0$  or 
 $\lambda = w\pi_0$ for some unit $w \in R_0$.  
Let    $D \in {}_2\text{Br}(F)$ be a division algebra  with  $F/F_0$-involution $\sigma$.  
In this section we show that  if ind$(D) = 4$, then 
  $D$ is   a tensor product of  two quaternion  algebras with some properties.   
  Suppose  that for any central simple algebra  $A\in  {}_2\text{Br}(k)$,   ind$(A) \leq 2$. Then we show that ind$(D) \leq 8$. 
  Further we show that  if  ind$(D) = 8$, then $D$  is isomorphic to a  tensor product of three quaternion algebras with some properties.  
  
  Let $R$ be the integral closure of $R_0$ in $F$. 
By the assumption on $\lambda$, $R$ is a 2-dimensional regular local ring with maximal ideal $\mathfrak{m} = (\pi,\delta)$ (\cite[Theorem 3.1, 3.2]{parimala-uinvariant}), where;
 if  $\lambda $ is a unit in $R_0$, then $\pi = \pi_0$ and $\delta = \delta_0$ and  if $\lambda = w\pi_0 $, then $\pi = \sqrt{\lambda}$ and $\delta = \delta_0$. 

 \begin{prop}
\label{ramified}
 Let $R$ and $F$ be as above.  Let  $D \in {}_2 \text{Br}(F)$ with an $F/F_0$ involution. 
 Suppose that $D$ is unramified on $R$ except possibly at $(\pi)$ and $(\delta)$.  
 If $\lambda = w\pi_0$ for some unit $w \in R_0$, then $D =  (D_0 \otimes (v_0, \delta_0)) \otimes_{F_0} F$ for some 
$D_0 \in   {}_2\text{Br}(F_0)$  which is unramified on $R_0$ and $v_0 \in R_0$ a unit. 
\end{prop}

\begin{proof}  Since $F/F_0$ is ramified at $\pi_0$,  by (\cite[Lemma 6.3]{parimala2020localglobal}), $D$ is unramified at $\pi$.
Hence, by(\cite[Proposition 3.5]{wu}),  $D  = D' \otimes  (v,  \delta)$ for some   unit $v \in R$  and $D'$ unramified on $R$.
Let $\Gamma'$ be an Azumaya $R$-algebra with $\Gamma'  \otimes F \simeq D'$.
Since $R/\mathfrak{m} = R_0/\mathfrak{m}_0$ and $R$ is complete, there exists an Azumaya $R_0$-algebra $\Gamma_0$  with 
$\Gamma'  \simeq \Gamma_0 \otimes R$.  Let $D_0 = \Gamma_0 \otimes F_0$. 
Since $R/\mathfrak{m}  =  R_0/\mathfrak{m}_0$, we have $v = v_0 v_1^2$ for some $v_1 \in R_0$ a unit.
Since $\delta = \delta_0$, we have 
$D  =  (D_0 \otimes (v_0, \delta_0)) \otimes F$ as required. 
\end{proof}
 
 For the rest of the section, we assume that $\lambda   \in R_0$ is a unit. 
 In particular $\pi= \pi_0$ and $\delta = \delta_0$ and $F/F_0$ is unramified on $R_0$. 
 Let $\tau$ denote the non trivial automorphism of $F/F_0$. 
 
 \begin{prop}
 \label{components-invol}
  Let $\Gamma'$ be an Azumaya $R$-algebra and $D' = \Gamma'  \otimes F$. Suppose that  $D' \in \Brt(F)$.
 Let $D = D' \otimes (u, \pi) \otimes (v, \delta)$  (resp. $ D =  D'  \otimes (u\pi, v\delta))$.
 If $D$ has a $F/F_0$-involution, then $D'$, $(u, \pi)$ and $(v, \delta)$  (resp. $D'$ and  $(u\pi, v \delta)$) have 
 $F/F_0$-involutions.  
 \end{prop}
 
 \begin{proof} Suppose $D = D' \otimes (u, \pi) \otimes (v, \delta)$ has a $F/F_0$-involution. 
 Since cores$_{F/F_0}(D) = 0$,   by (\cite[Lemma 6.4]{parimala2020localglobal}),   
 cores$_{F_\delta/F_{0\delta_0}}(D'\otimes (u, \pi)) = 0$  and cores$_{F_\delta/F_{0\delta_0}}(v, \delta) = 0$.
 Since $D'$ is unramified on $R$, cores$_{F/F_0}(D')$ is unramified on $R_0$. 
 Since $\pi = \pi_0 \in R_0$  and $\delta= \delta_0 \in R_0$, 
 cores$_{F /F_{0 }}(u, \pi) = (N_{F/F_0}(u), \pi_0) \otimes F_{0}$ and
 cores$_{F/F_0}(v, \delta) = (N_{F/F_0}(v), \delta_0) \otimes F_{0}$.
 In particular  cores$_{F /F_{0 }} ( D' \otimes (u, \pi))$ and cores$_{F/F_0}(v,\delta)$ are unramified on 
 $R_0$ except possibly at $(\pi_0)$. 
  Hence, by (\ref{reddy}), cores$_{F/F_0}(D' \otimes (u, \pi) ) = 0$  and cores$_{F /F_0}(v, \delta) = 0$.
 The same argument implies that cores$_{F/F_0}(D' ) = 0$  and cores$_{F /F_0}(u, \pi) = 0$.
 Hence $D'$, $(u, \pi)$ and $(v, \delta)$ have $F/F_0$-involutions. 
 
 The case $D = D' \otimes (u\pi, v\delta)$ is similar. 
 \end{proof}
  
  \begin{lemma}
  \label{decent} Let 
  $D_1 = (u, \pi)$ (resp. $(u\pi, v \delta)$,  $(u, v)$)  for some units $u, v \in R$. If $D_1$ has a $F/F_0$ involution, then 
  $D_1 = (u_0, \pi)$ (resp. $(u_0\pi, v_0 \delta)$,  $(u_0, v_0)$) for some $u_0, v_0 \in R_0$ units. 
  \end{lemma}

  \begin{proof} Suppose $D_1 = (u, \pi)$ for some $u \in R$ unit. 
  Since $\pi =\pi_0 \in R_0$, we have cores$_{F/F_0}(u, \pi) = (N_{F/F_0}(u),\pi_0)$.
  Since $D_1$ has a $F/F_0$-involution, we have  $(N_{F/F_0}(u),\pi_0) = 0 \in Br(F_0)$. 
  Since  the residue of $(N_{F/F_0}(u), \pi_0)$ at $\pi_0$ is the image of $N_{F/F_0}(u)$ in $\kappa(\pi_0)^*/\kappa(\pi_0)^{*2}$,  
  the image of $N_{F/F_0}(u)$ is a square in $\kappa(\pi_0)$. Since $R_0$ is a complete local  ring with $\pi_0$ a regular prime,
  $N_{F/F_0}(u)$ is a square in $R_0$. Let $w_0 \in R_0$ be such that 
  $N_{F/F_0}(u) = w_0^2$.  Then  $N_{F/F_0}(w_0^{-1}u) = 1$ and hence 
  $w_0^{-1}u = \theta \tau(\theta)^{-1}$ for some $\theta \in R$.
  We have   $ u  \tau(\theta)^2 = w_0  \theta \tau(\theta) = u_0 \in R_0$ and $(u, \pi) = (u_0, \pi)$. 
  
  Suppose $D_1 = (u\pi, v \delta)$. As above, by taking the residues at $\pi$ and $\delta$, we see that $N_{F/F_0}(u)$ and 
  $N_{F/F_0}(v)$ are squares. Hence as above, we can replace $u$ and $v$ by $u_0$ and $v_0$ for some $u_0, v_0 \in R_0$ units.   
  
  Suppose that $D_1 = (u,  v)$ for some $u, v \in R$. Since $D_1$  has an $F/F_0$-involution and 
  $D_1$ is unramified on $R$, $D_1 = D_0 \otimes F$ for some quaternion algebra $D_0$ over $F_0$ which is
  unramified on $R_0$ (\cite[Theorem 7.4]{auslander}).  In particular $D_0 = (u_0, v_0)$ for some units $u_0, v_0 \in R_0$
  \end{proof}

    \begin{cor}
  \label{disc-2dim-second2}
   Let $D \in {}_2Br(F)$ with an $F/F_0$-involution. Suppose that $D$ is unramified on $R$ except possibly at $(\pi)$ and $(\delta)$ and ind$(D) =2$.
  Then one of the following holds \\
  i)  $D$ is unramified on $R$ \\
  ii)  $D \simeq   (u_0,  u_1\pi_0)$ or $D  \simeq  (v_0, v_1\delta_0)$   \\
  iii) $D \simeq (  u_0,  u_1\pi_0\delta_0)$    \\
 iv)  $D \simeq   (u_0\pi_0, v_0\delta_0)$  \\
  for some units $w_i, u_i, v_i \in R_0$
  \end{cor} 
  
  \begin{proof}  Follows from (\ref{wu-disc}),  (\ref{components-invol}) and (\ref{decent}).   
  \end{proof}
  
  \begin{cor}
  \label{disc-2dim-second}
   Let $D \in {}_2Br(F)$ with an $F/F_0$-involution. Suppose that $D$ is unramified on $R$ except possibly at $(\pi)$ and $(\delta)$ and ind$(D) =4$.
  Then one of the following holds \\
  i)  $D$ is unramified on $R$ \\
  ii)  $D \simeq (w_0, w_1)  \otimes (u_0,  u_1\pi_0)$ or $D  \simeq (w_0, w_1)\otimes (v_0, v_1\delta_0)$   \\
  iii) $D \simeq (w_0, w_1)  \otimes (u_0,  u_1\pi_0\delta_0)$    \\
  iv)  $D \simeq (u_0, u_1\pi_0) \otimes (v_0, v_1\delta_0)$ \\
  v)  $D \simeq (w_0, w_1) \otimes (u_0\pi_0, v_0\delta_0)$  \\
  for some units $w_i, u_i, v_i \in R_0$
  \end{cor} 
  
  \begin{proof}  Follows from (\ref{disc-2dim}),  (\ref{components-invol}) and (\ref{decent}).   
  \end{proof}
  
\begin{cor}
  \label{disc-2dim-second8}
   Let $D \in {}_2Br(F)$ with an $F/F_0$-involution. Suppose that $D$ is unramified on $R$ except possibly at $(\pi)$ and $(\delta)$ and
   every element of  ${}_2Br(k)$ has index at most 2. If ind$(D) = 8$, then 
  Then $D \simeq  (w_0, w_1)  \otimes (u_0, \pi_0) \otimes (v_0, \delta_0)$  
   for some units $w_0, w_1, u_0, v_0\in R_0$. 
  \end{cor} 
  
  \begin{proof}  By the assumptions on $k$ and $D$, by (\ref{deg8}), 
  $D \simeq (w_0, w_1) \otimes (u_0, \pi) \otimes (v_0, \delta)$  
   for some units $w_0, w_1, u_0, v_0\in R$. Since $D$ has a $F/F_0$-involution, by  (\ref{components-invol}), 
   $(w_0, w_1)$, $(u_0, \pi)$ and $(v_0, \delta)$ have $F/F_0$-involutions. 
  As in the proof of (\ref{disc-2dim-second}),  we can assume $w_0, w_1, u_0, v_0  \in R_0$.
  \end{proof}

\section{maximal orders}
\label{maximal-order}

Let $R$ be a complete  regular local  ring   with residue field $k$, $(\pi, \delta)$ maximal ideal  and field of fractions 
$F$.   Suppose that char$(k) \neq 2$. Let $D\in {}_2\text{Br}(F)$   be a division algebra which is unramified on $R$ except possibly at  $(\pi)$ and $(\delta)$. 
By ((\cite{wu}, Proposition 3.5)), we know that $D = D_0 \otimes D_1$  for some $D_0 \in {}_2\text{Br}(F)$ which is unramified on $R$ and 
$D_1$ is $(u,  v\pi)$ or $(u,  v\delta)$ or $(u, w\pi) \otimes (v, t\delta)$ or   $(u, v\pi\delta)$  or $(u\pi, v\delta)$ for some units $u, v \in R$.
If  $D\simeq D_0 \otimes D_1$, then   in this section we show that there is a maximal $R$-order with some properties. 

 For an integral domain $R$ and  $a, b \in R$ non zero elements, 
let $R(a, b)$ be the $R$-algebra generated by $i, j$ with $i^2 = a$, $j^2 = b$ and $ij = -ji$. 
Suppose that $2 \in R$ is a unit.
Then $R(a, b)$ is a $R$-order in the quaternion algebra $(a, b)$ over $F$. 
Further note that if $a, b \in R$ are units, then $R(a, b)$ is an Azumaya $R$-algebra.

\begin{prop}
\label{maximal-order-cdvr}
Let $R$ be a complete discrete valuation ring   with residue field $k$ and field of fractions 
$F$.   Suppose that char$(k) \neq 2$.  Let $\Gamma_0$ be an Azumaya algebra over $R$ and 
$D_0 = \Gamma_0 \otimes_R K$.  Let $u \in R$ be a unit and $\pi \in R$ a parameter. 
If $D = D_0 \otimes (u, \pi)$ is a division algebra, then $\Gamma = \Gamma_0 \otimes_R R(u, \pi)$  is 
the  maximal $R$-order of $D$. 
\end{prop}

\begin{proof}  Suppose that $D$ is division.  Let $d = $ deg$(D)$ and $d_0 = $ deg$(D_0)$. Then $d=  2 d_0$. 

There is a discrete valuation on $D$
given by $\nu_D(z) =  \nu( Nrd_D(z))$ (\cite[139]{reiner}).  Furthermore $\Gamma' = \{ z \in D^* \mid \nu_D(z) \geq 0 \} \cup \{ 0 \} = \{z  \in D \mid 
z  ~\text{is ~integral~ over} ~ R \} $  is the unique maximal $R$-order of $D$ (\cite[Theorem 12.8]{reiner}). 

Since $\Gamma_0$ and $R(u, \pi)$ are finitely generated $R$-modules,     $\Gamma$ is a finitely generated $R$-module and hence every element of
$\Gamma$ is  integral over $R$. Hence $\Gamma \subseteq \Gamma'$.  We now show that $\Gamma' \subseteq \Gamma$.

Let $i, j \in (u, \pi)$ be the standard generators with $i^2 = u$, $j^2 = \pi$ and $ij = -ji$.
Let $D_ 1= D_0 \otimes F(i) \subset D$ and $\Gamma_1 = \Gamma_0 \otimes R[i]$. 
Then   $D = D_1 + D_1j$ and $\Gamma = \Gamma_1 \oplus \Gamma_1 j$. 
Since $D_0$ is unramified, $D_1$ is unramified. Since $D$ is a division algebra, $D_1$ is a division algebra.
Since $\Gamma_0$ is an Azumaya algebra over $R$,  $\Gamma_1$ is  an Azumaya algebra over $R[i]$. In particular 
$\Gamma_1$ is the maximal $R[i]$-order in $D_1$. 
Since  $F(i)/F$ is an unramified extension and  $D_0$ is unramified on $R$, $\pi$ is a parameter in $D_1$. 
Therefore, $\nu_D(z)$ is a multiple of  $2d_0$  for all $z \in D_1$ (\cite[139]{reiner}).
Since $j^2= \pi$ and $Nrd_D(j) = \pi^{d_0}$, we have $\nu_D(j) = d_0$. 

Let $z \in \Gamma'$. Then $z = z_1 + z_2j$ for some $z_1, z_2 \in D_1$. 
Suppose that $\nu_D(z_1) =  \nu_D(z_2j)$.  Then $\nu_D(z_1) - \nu_D(z_2) = \nu_D(j) = d_0$.
This is a contradiction, since  $\nu_D(z_1)$ and $\nu_D(z_2)$ are multiple of $2d_0$. 
Hence  $\nu_D(z_1) \neq \nu_D(z_2j)$. Then $\nu_D(z) = \text{min}\{\nu_{D}(z_1), \nu_{D}(z_2j)\} \geq 0$. 
In particular $\nu_D(z_1) \geq 0$ and hence $z_1 \in \Gamma_1$. 
Since $\nu_D(z_2j) = \nu_D(z_2) + \nu_D(j) = \nu_D(z_2) +d_0 $, 
we have $\nu_D(z_2) \geq -d_0$. Since $\nu_D(z_2) $ is a multiple of $2d_0$, it follows that 
$\nu_D(z_2) \geq 0$ and hence $z_2 \in \Gamma_1$. Thus $z \in \Gamma$.
\end{proof}

\begin{prop}
\label{maximal-order-dvr}
Let $R$ be a   discrete valuation ring   with residue field $k$,  field of fractions 
$F$ and   $\hat{F}$ the completion of $F$.
Suppose that char$(k) \neq 2$.  Let $\Gamma_0$ be an Azumaya algebra over $R$ and 
$D_0 = \Gamma_0 \otimes_R F$.  Let $u \in R$ be a unit and $\pi \in R$ a parameter. 
If $ (D_0 \otimes_F(u, \pi) ) \otimes_F\hat{F}$ is a divison algebra, then $  \Gamma_0 \otimes_R R(u, \pi)$ is 
the maximal $R$-order of $D_0 \otimes_F(u, \pi)$. 
\end{prop}

\begin{proof} 
Let $\hat{R}$ be the completion of $R$. 
Let $\hat{\Gamma}_0 = \Gamma_0 \otimes \hat{R}$.  Then $\hat{\Gamma}_0$ is an Azumaya algebra over $\hat{R}$.
 Since $ (D_0 \otimes_F(u, \pi) ) \otimes_F\hat{F}$ is a divison algebra, 
  by  (\ref{maximal-order-cdvr}), $\hat{\Gamma}_0 \otimes \hat{R}(u, \pi)$ is a maximal $\hat{R}$-order of $ (D_0 \otimes_F(u, \pi) ) \otimes_F\hat{F}$. 
  Thus,  by  (\cite{reiner}, Theorem 11.5), $  \Gamma_0 \otimes_R R(u, \pi)$ is 
the maximal $R$-order of $D_0 \otimes_F(u, \pi)$.  
\end{proof}

\begin{cor}
\label{maximal-order-2dim}
Let $R$ be a two dimensional  complete regular local ring with residue field $k$, field of fractions 
$F$ and maximal ideal $\mathfrak{m} = (\pi,\delta)$.   For units $u, v \in R$, let $D_1$ and $\Gamma_1$ be one of the following: \\
i) $D_1 = (u,v)$, $\Gamma_1 = R(u,v)$ \\
ii) $D_1 = (u, \pi)$, $\Gamma_1 = R(u, \pi)$ \\
iii) $D_1 = (\pi, \delta)$, $\Gamma_1 = R(\pi, \delta)$\\
iv) $D_1 = (u, \pi\delta)$, $\Gamma_1 = R(u, \pi\delta)$\\
v) $D_1 = (u, \pi) \otimes (v, \delta)$, $\Gamma_1 = R(u,\pi) \otimes R(v, \delta)$.\\
Let $\Gamma_0$ be an Azumaya algebra over $R$ and $D_0 = \Gamma_0 \otimes_RF$. 
If $  D_0\otimes_F D_1$ is a division algebra, then $ \Gamma = \Gamma_0 \otimes_R \Gamma_1$ is a maximal $R$-order of $D_0\otimes_FD_1$.  
\end{cor}

\begin{proof}
An order of a Noetherian integrally
closed domain is maximal if and only if it is reflexive and its localization at all height
one prime ideals are maximal orders (\cite{reiner}, Theorem 11.4). Since $\Gamma$ is a finitely generated free module, it is reflexive. Furthermore, $R$ is a regular local ring, hence it is Noetherian and integrally closed. We only need to show that $\Gamma_P$ is a maximal $R_P$-order for all height one prime ideals $P$.

Suppose that $D =D_0\otimes_FD_1 $ is a division algebra. 
Let $P$ be a height one prime ideal of $R$.  Suppose $P \neq \langle \pi \rangle, \langle \delta \rangle$.  Since $u, v \in R$ are units,
$u, v, \pi, \delta$ are units in $R_P$ and hence $\Gamma_1 \otimes R_P$ is an Azumaya $R_P$-algebra. In particular   $\Gamma \otimes R_P$ is an  Azumaya $R_P$-algebra. 
Hence $\Gamma_P$ is a   maximal $R_P$-order of $D$.  Suppose that $P  =  \langle \pi \rangle, \langle \delta \rangle$. 
 
\begin{enumerate}
    \item[i)]  Since   $u, v \in R$ are units,   ${(\Gamma_1)}_P$ is an Azumaya algebra over $R_P$. 
    Hence  $\Gamma_P$  is a maximal $R_P$-order on $D$.
    \item[ii)]  If $P \neq \langle \pi \rangle)$, then $\Gamma_1$ is an Azumaya $R_P$-algebra and hence 
    $\Gamma_P$ is a maximal $R_P$-order on $D$. 
    If $P = \langle \pi \rangle$, then $\Gamma_P$ is a maximal $R_P$-order on $D$ by  (\ref{maximal-order-dvr}).
    
    \item[iii), iv)]   If $P = \langle \pi \rangle$ or $P = \langle \delta \rangle$, then $\Gamma_P$ is a maximal $R_P$-order on $D$ by  (\ref{maximal-order-dvr}).
        
    \item[v)]   Suppose  $P = \langle \pi \rangle$. Let $\Gamma_0' = \Gamma_0 \otimes R_P(v ,\delta)$.
    Since $v, \delta$ are units in $R_P$, $R_P(v, \delta)$ is an Azumaya $R_P$-algebra.
    Since $D = (D_0 \otimes (v, \delta)) \otimes (u, \pi)$ and $\Gamma = (\Gamma_0 \otimes R_P(v, \delta)) \otimes R_P(u, \pi)$, 
    by (\ref{maximal-order-dvr}), $\Gamma_P$  is a maximal $R_P$-order on $D$.  If $P = \langle \delta \rangle$, a similar argument holds. 
   
\end{enumerate}
\end{proof}

\section{A local global principle for hermitian forms over two dimensional local fields}
\label{lgp-local}

Let $R_0$ be a 2-dimensional complete regular local ring with maximal ideal $\mathfrak{m}_0 = (\pi_0,\delta_0)$ and residue field $k_0$.
Suppose that char$(k_0) \neq 2$. 
Let $F_0$ be the field of fractions of $R_0$ and let $F = F_0(\sqrt{\lambda})$ be an extension of degree at most  2, 
with $\lambda $ a unit in $R_0$  or 
a unit times $\pi_0$.  
Let    $D \in {}_2\text{Br}(F)$ be a division algebra  with  $F/F_0$-involution $\sigma$  and $h$ an hermitian form over $(D, \sigma)$. 
In this section, under some assumptions on $D$, $F_0$ and $h$, we  prove that if $h\otimes F_{0\pi}$ or $h \otimes F_{0\delta}$ is isotropic, then 
$h$ is isotropic. \\

Let $R$ be the integral closure of $R_0$ in $F$. 
By the assumption on $\lambda$, $R$ is a 2-dimensional regular local ring with   maximal ideal $(\pi,\delta)$ (\cite[3.1, 3.2]{parimala-uinvariant}), where;
 if  $\lambda $ is a unit in $R_0$, then $\pi = \pi_0$ and $\delta = \delta_0$ and  if $\lambda $ is a unit times $\pi_0$, then $\pi = \sqrt{\lambda}$ and $\delta = \delta_0$. \\
 
Let $G(D, \sigma, h) = SU(D, \sigma, h)$ if $F = F_0$ and $G(D, \sigma, h) = U(A, \sigma, h)$ if $[ F : F_0] = 2$. \\

We begin with the following, which is  proved by Wu (\cite[Corollary 3.12]{wu}) for $D$ a quaternion algebra. 
 
  \begin{prop}
 \label{local1}  Let $F_0$ and  $F$  be as above. Let $D\in {}_2\text{Br}(F)$ and $\sigma$ an $F/F_0$- involution.
 Suppose that $D$ is a division algebra which  unramified on $R$ except possibly at $(\pi)$ and $(\delta)$.  Let $d = $deg$(D)$, 
 $e_0$ the ramification index of $D$ at $\pi$ and $e_1$ the ramification index of $D$ at $\delta$. 
 Suppose that 
 there exists a maximal $R$-order $\Gamma$ of $D$ and $\pi_D, \delta_D \in \Gamma$ such that 
 $\sigma(\pi_D)= \pm \pi_D$, $\sigma(\delta_D) = \pm \pi_D$ and $\pi_D\delta_D = \pm \delta_D \pi_D$ and
 Nrd$(\pi_D) = \theta_0 \pi^{\frac{d}{e_0}}$ and  Nrd$(\delta_D) = \theta_1 \delta^{\frac{d}{e_1}}$ for some units $\theta_0, \theta_1 \in R$.
 Let $h = \langle a_1, \cdots, a_n\rangle$  be an hermitian form over $(D, \sigma)$. 
 Suppose that  for $1\leq i \leq n$, $a_i \in \Gamma$ and Nrd$(a_i)$ is a  product of a unit in $R$,  a power of $\pi$ and a power of 
 $\delta$.  If $h \otimes F_\pi$ or $h \otimes F_\delta$ is isotropic, then $h$ is isotropic.  
 \end{prop}

  \begin{proof}   Follows from (\cite[Corollary 3.3]{wu}). 
  \end{proof}
 
 As a consequence we have the following  (cf. \cite[Corollary 3.12]{wu})

 \begin{prop}
 \label{local2}  Let $F_0$ and  $F$  be as above. Let $D\in {}_2\text{Br}(F)$ and $\sigma$ an $F/F_0$- involution.
 Suppose that $D$ is a division algebra which is unramified on $R$ except possibly at $(\pi)$ and $(\delta)$.  Let $d = $deg$(D)$, 
 $e_0$ the ramification index of $D$ at $\pi$ and $e_1$ the ramification index of $D$ at $\delta$. 
 Suppose that 
 there exists a maximal $R$-order $\Gamma$ of $D$ and $\pi_D, \delta_D \in \Gamma$ such that 
 $\sigma(\pi_D)= \pm \pi_D$, $\sigma(\delta_D) = \pm \pi_D$, $\pi_D\delta_D = \pm \delta_D \pi_D$,
 Nrd$(\pi_D) = \theta_0 \pi^{\frac{d}{e_0}}$ and  Nrd$(\delta_D) = \theta_1 \delta^{\frac{d}{e_1}}$ for some units $\theta_0, \theta_1 \in R$.
 Let $h = \langle a_1, \cdots, a_n\rangle$  be an hermitian form over $(D, \sigma)$. 
 Suppose that  for $1\leq i \leq n$, $a_i \in \Gamma$ and Nrd$(a_i)$ is a  product of a unit in $R$,  a power of $\pi$ and a power of 
 $\delta$. 
 Let $X$  be a projective homogeneous space under $G(D, \sigma, h)$. If $X(F_{0\pi}) \neq \emptyset$ or $X(F_{0\delta}) \neq \emptyset$,
 then $X(F_0) \neq \emptyset$. 
 \end{prop}

  \begin{proof}   
 Since the period of $D$ divides 2  by  our assumption on $D$, we have ind$(D) =  $ ind$(D\otimes F_\pi)$ (\ref{reddy}).
 Suppose that $X(F_{0\pi}) \neq \emptyset$ or $X(F_{0\delta}) \neq \emptyset$.  
 Using, (\ref{local1}),  the rest of the proof of   (\cite[Corollary 3.12]{wu}) can be applied here to  show that $X(F_0) \neq \emptyset$. 
  \end{proof}

  We fix the following. 
\begin{notation}
\label{maximal-order-inv}
Let $R_0$ be a 2-dimensional complete regular local ring with maximal ideal $\mathfrak{m}_0 = (\pi_0,\delta_0)$ and residue field $k_0$.
Suppose that char$(k_0) \neq 2$. 
Let $F_0$ be the field of fractions of $R_0$ and let $F = F_0(\sqrt{\lambda})$ be an extension of degree at most  2, 
with $\lambda $ a unit in $R_0$  or 
a unit times $\pi_0$. 
Let $R$ be the integral closure of $R_0$ in $F$. 
By the assumption on $\lambda$, $R$ is a 2-dimensional regular local ring with   maximal ideal $(\pi,\delta)$ (\cite[3.1, 3.2]{parimala-uinvariant}), where;
 if  $\lambda $ is a unit in $R_0$, then $\pi = \pi_0$ and $\delta = \delta_0$ and  if $\lambda $ is a unit times $\pi_0$, then $\pi = \sqrt{\lambda}$ and $\delta = \delta_0$. 
 Let   $ u_i, v_i \in R_0$ be units and $D_1$,    $\Gamma_1$   and $\sigma_1$  denote  one of the following: \\
 i)  $D_1 = F_0$,  $\Gamma_1 = R$, $\sigma_1 = id$ \\
  ii)  $D_1 =  (u_0,  u_1\pi_0)$,  $\Gamma_1 = R(u_0, u_1\pi_0)$,  $\sigma_1$ the canonical involution   \\
   iii)  $D_1 = (u_0\pi_0, v_0\delta_0)$, $\Gamma_1 = R(u_0\pi_0, v_0\delta_0)$,  $\sigma_1$ the canonical involution \\
  iv)  $D_1  =  (u_0, u_1\pi_0) \otimes (v_0, v_1\delta_0)$, $\Gamma_1 = R(u_0, u_1\pi_0) \otimes R(v_0, v_1\delta_0)$, 
   $\sigma_1$ the  tensor product of the canonical involutions  \\
  Let $\Gamma_0$ be an Azumaya algebra over $R$ with an $R/R_0$- involution $\tilde{\sigma}_1$. Let $D_0 \simeq \Gamma_0 \otimes F$ and 
$\sigma_0 = \tilde{\sigma_0} \otimes 1$. 
Let $D \simeq D_0 \otimes D_1$, $\Gamma = \Gamma_0 \otimes \Gamma_1$ and $\sigma = \sigma_0 \otimes \sigma_1$.
Then $\sigma$ is a $F/F_0$-involution on $D$ and $\sigma(\Gamma) = \Gamma$. 
Let $d_i$ denote the degree of $D_i$. 
The following table gives a choice of $\pi _{D}, \delta_{D} \in \Gamma$ and some of their  properties. 
\begin{table}[H]
\begin{tabular}{@{}|l|l|l|l|l|l|l|l|@{}}
\toprule
 D& $\pi_D$ & $\delta_D$ & $\text{Nrd}(\pi_D)$ & $\text{Nrd}(\delta_D)$ & $\sigma(\pi_D)$ & $\sigma(\delta_D)$ & $\sigma(\pi_D\delta_D)$ \\ \midrule
 
 $D_0 $ & $ \pi_0$ &$  \delta_0$  & $\pi_0^{d_0}$  & $\delta_0^{d_0}$  & $\pi_D$ & $\delta_D$ &$\pi_D\delta_D$  \\ \midrule
 
 $D_0 \otimes (u_0,  u_1\pi_0)$& $1 \otimes j$ &$1 \otimes \delta$  & $(u_1\pi_0)^{d_0}$  & $\delta_0^{2d_0}$  & $-\pi_D$ 
 & $\delta_D$ &$-\pi D\delta_D$  \\ \midrule

 $D_0 \otimes  (u_0\pi_0, v_0\delta_0) $& $1 \otimes i$   &  $1 \otimes j$  &$(u_0\pi_0)^{d_0}$  
 &$(v_0\delta_0)^{d_0}$  &$-\pi_D$  &$-\delta_D$  & $-\pi_D\delta_D$  \\ \midrule
 
 $D_0 \otimes (u_0, u_1\pi_0) \otimes (v_0, v_1\delta_0)$& $1 \otimes j_1 \otimes 1$ & $1 \otimes 1 \otimes j_2$   &
 $(u_1\pi_0)^{2d_0}$  &$(v_1\delta_0)^{2d_0}$  &$-\pi_D$  &$-\delta_D$  &$ \pi_D\delta_D$  \\ \bottomrule
 
\end{tabular}
\end{table} 
\end{notation}

\begin{cor}\label{local-results-1}
Let $F$, $D$, $\sigma$ and $\Gamma$  be as in (\ref{maximal-order-inv}). 
Suppose that $D$  is a division algebra. 
Let $h = \langle a_1, \cdots, a_n\rangle$ be a hermitian form over $(D,\sigma)$ with $a_i \in \Gamma$. 
Suppose Nrd$(a_i)$ is a unit times a power of $\pi$ and a power of $\delta$. 
 Let $X$ be a projective homogeneous space under $G(D,\sigma,h)$ over $F_0$. If $X(F_{0\pi_0}) \neq \emptyset$ or 
 $X(F_{0\delta}) \neq \emptyset$, then $X(F_0) \neq \emptyset$.
\end{cor}

\begin{proof} By (\ref{maximal-order-2dim}),  $\Gamma$ is a maximal $R$-order of $D$. 
 Let $e_0$ be the ramification index of $D$ at $\pi$ and $e_1$   be the ramification index of $D$ at $\delta$.
If $D_1$ as in (\ref{maximal-order-inv}(i)), then $e_0 = e_1 = 1$. If $D_1$ is as in (\ref{maximal-order-inv}(ii)), 
then $e_0 = 2$ and $e_1 = 1$. If 
$D_1$ as in (\ref{maximal-order-inv}(iii) or (iv)), then $e_0 = e_1 = 2$.  Let  $\pi_D$ and $\delta_D$ be as in (\ref{maximal-order-inv}).
Then  $\pi_D$ and $\delta_D$ satisfy the assumptions  of (\ref{local2}). 
Hence,  by (\ref{local2}),  $X(F_0) \neq \emptyset$.
\end{proof}

\section{Behavior under blowups }
\label{blowups}

Let $R_0$ be a 2-dimensional complete regular local ring with maximal ideal $\mathfrak{m}_0 = (\pi_0,\delta_0)$ and residue field $k_0$.
Suppose that char$(k_0) \neq 2$. 
Let $F_0$ be the field of fractions of $R_0$ and let $F = F_0(\sqrt{\lambda})$ be an extension of degree at most  2.
Let    $D \in {}_2\text{Br}(F)$ be a division algebra  with  $F/F_0$-involution $\sigma$  and $h$ an hermitian form over $(D, \sigma)$
and $G(A, \sigma, h)$  be as in (\S \ref{lgp-local}).     
 Let $X$ be a projective homogeneous variety under $G(A, \sigma, h)$ over $F_0$. 
  Suppose that $X(F_{0\nu}) \neq \emptyset$  for all  divisorial discrete valuations $\nu$ of $F_0$.
Under some assumptions on $D$, in this section we prove that there exists a sequence of blowups $\YY$ of Spec$(R_0)$ such that 
$X(F_{0P}) \neq \emptyset$  for all closed points $P$ of $\YY$.

Let $\XX_0 =  Proj(R_0[x, y]/(\pi_0 x  - \delta_0 y))$.
Let $Q_1$ and $Q_2$  be the closed points of $\XX_0$  given by the homogeneous ideals  $(\pi_0,\delta_0, y)$ and 
  $(\pi_0,\delta_0, x)$. 
Let $\tau$ be  the nontrivial automorphism of $F/F_0 $ if $F \neq F_0 $ and let $\tau$ be the identity if $F = F_0 $. 

We begin with the following. 
 
   \begin{lemma}
 \label{blowup-index2} 
Suppose that  $\lambda $ a unit in $R_0$  or 
a unit times $\pi_0$.  
Let $a, b \in R_0 $ be nonzero and square free. Suppose that 
the support of $a$ and $b$ is at most $\pi_0$ and $\delta_0$ and have no common factors.
 Then  for any closed point $P \in \XX_0$, 
there exist  $a', b', \pi', \delta'  \in {\OO}_P$   such that  
the maximal ideal at $P$ is generated by $\pi'$ and $\delta'$,  $a'$ and $b'$ are square free,  have no common factors, 
the support  is at most $(\pi')$ or $(\delta')$ and $ (a, b) \otimes F_{0P} =  (a', b') $ and  $R_0 (a, b) \subset \hat{\OO}_P(a', b')$. 
\end{lemma}

\begin{proof}
Suppose $a$ is a unit   $R_0 $. Then  $b = v_0$ or $v_0 \pi_0$ or $v_0\delta_0$ or $v_0\pi_0\delta_0$ for some
unit $v_0 \in R_0$.  If $b = v_0$ or $v_0 \pi_0$ or $v_0\delta_0$, then  it is easy to see that   
 $a' = a$ and $b' = b$  have the required properties.  Suppose $b = v_0\pi_0\delta_0$. 
 Suppose $P \neq Q_1$, $Q_2$. Then the maximal ideal at $P$ is given by  $(\pi_0, \delta')$ with 
 $\pi_0 = w'\delta_0$ for some unit $w'$ in $\OO_P$.  We have $(a, b) = (a, v_0\pi_0\delta_0) = (a, v_0w')$.
 In this case it is easy to see that $a' = a$ and $b' = v_0w'$ have the required property. 
 Suppose $P = Q_1$. Then the maximal ideal at $P$ is given by $(t, \delta_0)$ with $\pi_0 = t\delta_0$.
 We have $(a, b) = (a, v_0t \delta_0) $ and $a' = a$, $b' = v_0t\delta_0$ have the required properties.  
 The case $P = Q_2$ is similar.

 Suppose neither $a$ nor $b$ is a unit in $R_0$. Then  by the assumption on $a, b$, 
  we have $\{ a, b \} = \{u_0\pi_0 , v_0\delta_0\}$ for some units $u_0, v_0 \in R_0$. 
  Suppose $P = Q_1$. 
  Since   the maximal ideal of  
$\hat{\OO}_{Q_1}$ is given by $(t, \delta_0)$ with $\pi_0 = t\delta_0$,   we have
$(a, b) \otimes F_{0Q_1}  = (u_0\pi_0, v_0\delta_0) \otimes F_{0Q_1}   = (u_0t\delta, v_0\delta) \otimes F_{0Q_1}
\simeq (-u_0v_0t, v_0\delta_0)$.  It is easy to see that  $a' = -u_0v_0t$ and $b' = v_0\delta_0$ have the required properties. 
 The case $P = Q_2$ is similar.  Suppose $P \neq Q_1, Q_2$. 
Since  the maximal ideal at $P$ is given by  $(\pi_0, \delta')$ with 
 $\pi_0 = w'\delta_0 $ for some unit $w'$ in $\OO_P$ and  $(a, b) = (u_0\pi_0, v_0 \delta_0) = (u_0\pi_0, v_0w'\pi_0) = (u_0\pi_0,
 v_0w'u_0)$,
 $a' = u_0\pi_0$ and $b' = v_0w'u_0$ have the required properties.  
\end{proof}

\begin{lemma} 
\label{blowuppoint-index4}
Suppose that  $\lambda $ a unit in $R_0$  or 
a unit times $\pi_0$.  
 Suppose $D  =  (u_0, u_1\pi_0) \otimes (v_0, v_1\delta_0)$ is a division algebra  
for some units $u_i, v_i \in R _0$. 
Let $\sigma$ be the tensor product of canonical involutions on $(u_0, u_1\pi_0)$,  $(v_0,  v_1\delta_0)$.
 Then for  $P = Q_1$ or $Q_2$,    there exist an isomorphism   $ \phi_P :   D \otimes F_{0P} 
  \to   (u_0',  u_1'\pi') \otimes (v_0',  v_1'\delta')$ 
 for some $u_i' , v_i' $ units in   the local ring  ${\OO}_{P}$  at $P$ and  the maximal ideal of ${\OO}_{P}$  
 is given by $(\pi', \delta')$  such that  $ \phi_P(  R_0(u_0, u_1\pi)\otimes R_0(v_0, v_1\delta) )  \subset  {\OO}_{P}(u_0',  u_1'\pi')
\otimes  {\OO}_{P}(v_0', v_1'\delta')$. 
 Further if   $\sigma'$ is  the  tensor product of canonical involutions on $(u_0', u'_1\pi')$ and $(v_0', v_1'\delta')$, then   
 there exists $\phi_P \in  {\OO}_{P}(u_0',  u_1'\pi')
\otimes  {\OO}_{P}(v_0', v_1'\delta')$ such that 
 int$(\theta_{P}) \sigma'  =   \phi_P \sigma \phi_P^{-1} $  and  the 
support of  Nrd$(\theta_{P})$  at most $(\pi')$ and $(\delta')$. 
\end{lemma} 

\begin{proof} 
Since  $Q_1$ is the closed point given by the homogeneous ideal $(\pi_0,\delta_0, y)$,  the maximal ideal of 
${\OO}_{Q_1}$ is given by $(t, \delta_0)$ with $\pi_0 = t\delta_0$.
Thus we have $D \otimes F_{0P}  =  (u_0, u_1\pi_0) \otimes (v_0, v_1\delta_0) = 
(u_0, u_1t\delta_0) \otimes (v_0, v_1\delta_0) \simeq (u_0, u_1v_1^{-1}t) \otimes (u_0v_0,  v_1\delta_0)$. 

Let $i_1, j_1 \in (u_0, u_1\pi_0)$, $i_2, j_2 \in (v_0, v_1\delta_0)$,     $i_3, j_3 \in (u_0, u_1v_1^{-1}t)$ and $i_4, j_4 \in (u_0v_0,  v_1\delta_0)$ be the standard generators.
Then  we have an isomorphism $\phi_P :    (u_0, u_1\pi_0) \otimes (v_0, v_1\delta_0)     \to   (u_0, u_1v_1^{-1}t) \otimes (u_0v_0,  v_1\delta_0)$ given by 
 $\phi( i_1 \otimes 1 ) =  i_3 \otimes 1$,   $\phi( j_1 \otimes 1) =  j_3 \otimes j_4$,  $\phi ( 1 \otimes i_2) =  u_0^{-1}(i_3 \otimes i_4)$  and 
 $\phi( 1 \otimes j_2) =  1 \otimes j_4$.  Since $u_0$ is a unit in $R_0$,   $ \phi(  R_0(u_0, u_1\pi_0)
 \otimes R_0(v_0, v_1\delta_0) )  \subset  {\OO}_{Q_i}(u_0, u_1v_1^{-1}t)
\otimes  {\OO}_{Q_i}(u_0v_0, v_1\delta_0)$.  Let  $\theta_{Q_1}  = i_3 \otimes j_4$. Then it is easy to see that $\theta_{Q_1}$    has
 the required properties. 
  A similar   computation  gives the  required $\theta_{ Q_2}$. 
  \end{proof}
  
  \begin{lemma} 
  \label{curve-point-index4}
  Suppose that  $\lambda $ a unit in $R_0$  or 
a unit times $\pi_0$.  
  Suppose $D  =  (u_0, u_1\pi_0) \otimes (v_0, v_1\delta_0)$ is a division algebra  
for some units $u_i, v_i \in R _0$. 
Let $\sigma$ be the tensor product of canonical involutions on $(u_0, u_1\pi_0)$,  $(v_0,  v_1\delta_0)$.
Let $P \in \XX_0$ be a closed point not equal to $Q_1$ or $Q_2$. Then  there exists an isomorphism   $\phi_P : 
   D \otimes F_{0P}  \simeq (u_0', u_1') \otimes (v_0',  v_1'\pi')$ 
 for some $u_0', v_0', u_1', v_1'  $ units in   ${\OO}_P$    and  $m_P = (\pi', \delta')$ such that $ \phi_P( R_0(u, \pi)\otimes R_0(v, \delta) )  \subset  {\OO}_{P}(u_0', u_1')
\otimes  {\OO}_{P}(v_0', v_1'\pi)$. 
 Further if   $\sigma'$ is  the  tensor product of canonical involutions on $(u_0', u'_1)$ and $(v_0', v_1'\pi')$, then   
there exists   $\theta_{P} \in  {\OO}_{P}(u_0', u_1') \otimes  {\OO}_{P}(v_0', v_1'\pi)$   
such that int$(\theta_{P}) \sigma'  =   \phi_P \sigma \phi_P^{-1} $  and  the 
support of  Nrd$(\theta_{P})$  at most $(\pi')$ and $(\delta')$. 
\end{lemma} 

\begin{proof}
Since  $P$ is a closed point not equal to $Q_1$ or $Q_2$, the maximal ideal at $P$ is given by
$(\pi_0, \delta')$ and   $\delta_0 = w'  \pi_0 $  for some unit $w'$ in $\OO_{P}$. 
Thus we have $D  \otimes F_{0P} = (u_0, u_1\pi_0) \otimes (v_0, v_1\delta_0) = (u_0, u_1\pi_0) \otimes (v_0,  v_1 w' \pi_0)
 \simeq (v_0,  v_1 w' u_1^{-1}) \otimes (u_0v_0,  u_1\pi_0)$. 

Let $i_1, j_1 \in (u_0, u_1\pi)$,  $i_2, j_2 \in (v_0, v_1\delta_0)$, 
$i_3, j_3 \in (v_0, v_1 w' u_1^{-1})$ and $i_4, j_4 \in (u_0v_0,  u_1\pi_0)$ be the standard generators.
Then  we have an isomorphism $\phi_P  :    (u_0, u_1\pi_)) \otimes (v_0, v_1\delta_0) \to 
  (v_0,  v_1 w' u_1^{-1}) \otimes (u_0v_0,  u_1\pi_0)$ given by 
 $\phi( i_1 \otimes 1 ) =  v_0^{-1}(i_3 \otimes i_4)$, $\phi( j_1 \otimes 1) =  1 \otimes j_4$,   $\phi ( 1 \otimes i_2) =   (i_3 \otimes 1)$  and 
 $\phi( 1 \otimes j_2) =  j_3 \otimes j_4$.  
 Since $v_0 \in R_0$ is a unit, we have 
 $ \phi_P(  R_0(u_0, u_1\pi)\otimes R_0(v_0, v_1\delta) )  \subset  {\OO}_{P}(v_0,  v_1 w' u_1^{-1}) 
\otimes  {\OO}_{P}(u_0v_0,  u_1\pi)$. Let  $\theta_{P}  = i_3 \otimes j_4$. Then $\theta_{P}$ has
 the required properties. 
   \end{proof}

 The following  two results are   extracted from (\cite[\S 4]{wu}).
 
  \begin{lemma} 
 \label{notiii2}
  Let $F_0$ and $F$ be as above.  Let $D \in Br(F)$ be  a  quaternion division algebra over $F$ with 
 a $F/F_0$-involution $\sigma$. 
 Then there exists a sequence of blowups $\XX_0 \to Spec{R_0}$ such that, 
  the integral closure $\XX$ of $\XX_0$ in $F$ is regular and 
  ram$_{\XX}(D)$ is a union of regular curves with normal crossings.
  Further  for every closed point 
 $P$ of $\XX_0$ with   $D\otimes F_{0P} $ division, 
   $D\otimes F_{0P}  $ is  as in  (\ref{wu-disc} or \ref{disc-2dim-second2}) 
   and not of the type  (\ref{wu-disc}(iii) or  \ref{disc-2dim-second2}(iii)). 
  \end{lemma}
 
 \begin{proof}
There exists a sequence of blowups $\YY_0 \to Spec(R_0)$ such that 
the integral closure $\YY$ of $\YY_0$ is regular and 
  ram$_{\YY}(D)$ is a union of regular curves with normal crossings (cf. \cite[Corollary 11.3]{parimala2020localglobal}).
  Let $P$ be a closed point of $\YY_0$. 
  Since the integral closure of $\YY_0$ in $F$ is regular, the maximal ideal at $P$ is generated by 
  $(\pi_{P}, \delta_P)$ and  $F = F_0(\sqrt{\lambda_P})$ for some 
  $\lambda_P = u$ or $u\pi_P$ for some unit $u$ at $P$.  
  
  Suppose that   $D\otimes F_{0P}$ is division.  In particular $F \otimes F_{0P}$ is a field.
  Let $R_P$ be the integral closure of $\OO_P$ in $F$. Then $R_P$ is a regular two dimensional local ring 
  with maximal ideal $(\pi'_P, \delta_P')$ with $\delta'_P  = \delta_P$, 
  $\pi_P' = \pi_P$ if $\lambda_P$ is a unit in $\OO_P$  and $\pi_P' = \sqrt{\lambda_P}$ if $\lambda_P$ is not a 
  unit in $\OO_P$.  Further $D$ is unramified on $\OO_P$ except possibly at $(\pi_P')$ or 
  $(\delta_P')$.  In particular $D\otimes F_{0P}  $ is  as in  (\ref{wu-disc} or \ref{disc-2dim-second2}).

  Suppose $D\otimes F_{0P}$  as in (\ref{wu-disc}(iii) or  \ref{disc-2dim-second2}(iii)).    
  Note that there are only finitely many such closed points. 
  
 Suppose $F \otimes F_{0P}/ F_{0P}$ is ramified. Then, by (\ref{ramified}), we can assume that 
  $D\otimes F_{0P}$ is not of type (\ref{disc-2dim-second}(iii)).
  
  Suppose that $F\otimes F_{0P}/ F_{0P}$ is unramified.    Let $\XX_P \to Spec(\OO_P)$ be the simple blow up. 
  Then, it is easy to see that  for  every closed point  $Q$ of $\XX_P$, $D\otimes F_{0Q}$ 
  is not  of type
  (\ref{disc-2dim-second}(iii) or \ref{disc-2dim-second2}(iii)) (cf. \cite[Lemma 4.1]{wu}).
 \end{proof} 

 \begin{prop} 
 \label{index2}
 Let $D$ be  a central simple  algebra over $F$ with 
 a $F/F_0$-involution $\sigma$  and $h$ an hermitian form over $(D, \sigma)$.
  Let $X$ be a projective homogeneous variety under $G(D, \sigma, h)$ over $F_0$.
Suppose that ind$(D) \leq 2$. 
If  that $X(F_{0\nu}) \neq \emptyset$  for all  divisorial discrete valuations $\nu$ of $F_0$, 
  then  there exists a sequence of blowups $\YY \to Spec(R_0)$ such that for every  closed point $P$ of $\YY$,  
 $X(F_{0P}) \neq \emptyset$.  
 \end{prop}
 
 \begin{proof} By  Morita equivalence (\cite[Theorem 3.1,3.11,3.20]{knus-involutions} \& \cite[Chapter 1, 9.3.5]{knus-hermitian}), we assume that $D$ is  division.  If $D = F$, then $X(F_0) \neq \emptyset$ (\cite[Corollary 3.12]{wu}) and hence any blowup of 
 Spec$(R)$ has the required property.
 
 Suppose  ind$(D) = 2$.    Then $D \simeq D_0 \otimes_{F_0} F$ for some quaternion division algebra over $F_0$.
   Without loss of generality we assume that $\sigma$ is the canonical involution. 
 Then using (\ref{notiii2}), we get a sequence of blowups  $\YY_0$ of Spec$(R_0 )$ such that 
 for every closed point $P$ of $\YY$ with  $D \otimes F_{0P} $  division,  
  $D\otimes F_{0P} $     is  as in  (\ref{wu-disc} or \ref{disc-2dim-second2}) 
   and not of the type  (\ref{wu-disc}(iii) or  \ref{disc-2dim-second2}(iii)). 
 
 Suppose that $D \otimes F_{0P} $   is not division.  Suppose $F\otimes F_{0P}$ is a field.
 Since  ind$(D) \leq 2$, $D\otimes F_{0P}$ is a matrix algebra and hence by Morita equivalence, $h$ corresponds to a
 quadratic form over $F$. Thus  $X(F_{0P}) \neq \emptyset$ (\cite[Theorem 3.1]{patchingLGPpadic}). 
 Suppose $F\otimes F_{0P}$ is  not a field. Then $F\otimes F_{0P} \simeq F_{0P} \times F_{0P}$,
 $D \otimes F_{0P} \simeq D_0 \otimes F_{0P} \times D_0^{op} \otimes F_{0P}$ and $G(D, \sigma, h) \otimes F_{0P}  \simeq GL(M_n(D_0 \otimes F_{0P}))$ (\cite[p. 346]{knus-involutions}).
 Hence by (\ref{gla}), $X(F_{0P}) \neq \emptyset$. 
  
 Suppose that $D \otimes F_{0P} $   is  division.
 Then  
 $D\otimes F_P  = (a_p, b_P)$ for some $a_P, b_P$ as in (\ref{blowup-index2}).  
 In  particular $\OO_P(a_P, b_P) \otimes R$ is a maximal $\OO_P$-order of 
 $D\otimes F_{0P}  $.  We have $h\otimes F_{0P}  = \langle a_{1P}, \cdots , a_{nP}\rangle$ for some $a_{iP} \in \hat{\OO}_P(a, b)\otimes R$. 

Let $Y_0$ be the special fibre of $\YY_0$.  By (\ref{finite}), there exists a  finite subset $\PP$ of $Y_0$ such that 
$X(F_P ) \neq \emptyset $ for all $P \not \in \PP_0$.   Thus replacing $R $ by $\hat{\OO}_P$, we assume that 
$D  = (a, b)$ for some $a, b$ as in (\ref{blowup-index2}) and $h  = \langle a_1, \cdots , a_n\rangle$ for some $a_i \in R(a, b)$.

By (\cite[Lemma 4.2]{wu}), there exists a sequence of blowups  $\YY_1 \to Spec(R_0)$  
 such that the support of  Nrd$(a_i)$ is a union of regular curves with normal crossings
and for every closed point $Q$ of $\YY_1$ with  $D\otimes F_{0Q}$ is  division, 
$D\otimes F_{0Q}$  is  not of the form ( \cite[Lemma 3.6(5)]{wu}).  

Let $Q$ be a closed point of $\YY_1$.   If $D \otimes F_{0Q}$ is  a matrix algebra, then by (\cite[Corollary 4.7]{HHK-refinements}) and Morita equivalence, $X(F_{0Q}) \neq \emptyset$. 
Suppose $D \otimes F_{0Q}$ is  a division algebra.  Then,   by (\ref{blowup-index2}), $R (a, b)$ is contained in the 
corresponding maximal order  $\hat{\OO}_Q(a_Q, b_Q)$. 
Since  $h = \langle a_1 , \cdots , a_n \rangle$ with $a_i \in R (a, b)$ with support of Nrd$(a_i)$ is a union of regular curves with normal crossings, 
  by (\ref{local-results-1}), $X(F_{0Q}) \neq \emptyset$. 
 \end{proof}

 \begin{lemma} 
 \label{notiii}
  Let $D \in {}_2Br(F)$ be  a central simple  algebra over $F$ with 
 a $F/F_0$-involution $\sigma$  and $h$ an hermitian form over $(D, \sigma)$.
   Suppose that    ind$(D) = 4$. 
 Then there exists a sequence of blowups $\XX_0 \to Spec{R_0}$ such that, 
  the integral closure $\XX$ of $\XX_0$ in $F$ is regular and 
  ram$_{\XX}(D)$ is a union of regular curves with normal crossings.
  Further  for every closed point 
 $P$ of $\XX_0$ with ind$(D\otimes F_{0P} ) = 4$, 
   $D\otimes F_{0P} $ is  as in  (\ref{disc-2dim} or \ref{disc-2dim-second}) and not of the type 
    (\ref{disc-2dim}(iii) or \ref{disc-2dim-second}(iii)). 
  \end{lemma}
 
 \begin{proof}
 There exists a sequence of blowups $\YY_0 \to Spec(R_0)$ such that 
the integral closure $\YY$ of $\YY_0$ is regular and 
  ram$_{\YY}(D)$ is a union of regular curves with normal crossings (cf. \cite[Corollary 11.3]{parimala2020localglobal}).
   
  Let $P$ be a closed point of $\YY_0$.    In particular  if ind($D\otimes F_{0P}) = 4$, 
  then $D\otimes F_{0P}$ is as in  (\ref{disc-2dim} or \ref{disc-2dim-second}).
  Suppose $D\otimes F_{0P} = (u_P, v_P) \otimes (w_P,  w_P' \pi_P\delta_P)$ as in (\ref{disc-2dim}(iii)
   or \ref{disc-2dim-second}(iii)) for some 
  units $u_P, v_P, w_P, w_P' \in \OO_P$.  Note that there are only finitely many such closed points. 
  
  Suppose $F \otimes F_{0P}/ F_{0P}$ is ramified. Then, by (\ref{ramified}), we can assume that 
  $D\otimes F_{0P}$ is not of type (\ref{disc-2dim-second}(iii)).
  
  Suppose that $F\otimes F_{0P}/ F_{0P}$ is unramified.    Let $\XX_P \to Spec(\OO_P)$ be the simple blow up. 
  Then, it is easy to see that  for  every closed point  $Q$ of $\XX_P$, $D\otimes F_{0Q}$ 
  is not  of type  
  (\ref{disc-2dim}(iii) or \ref{disc-2dim-second}(iii)).
 \end{proof}

\begin{prop}
 \label{index4}
 Let $D \in {}_2Br(F)$ be  a central simple  algebra over $F$ with 
 a $F/F_0$-involution $\sigma$  and $h$ an hermitian form over $(D, \sigma)$.
  Let $X$ be a projective homogeneous variety under $G(D, \sigma, h)$ over $F_0$.
 Suppose that  ind$(D) \leq 4$.
 If $X(F_{0\nu}) \neq \emptyset$  for all  divisorial discrete valuations $\nu$ of $F_0$, 
 then  there exists a sequence of blowups $\YY \to Spec(R_0)$ such that for every  closed point $P$ of $\YY$,  
 $X(F_{0P}) \neq \emptyset$. 
  \end{prop}
  
 \begin{proof}  By (\ref{index2}), we assume that ind$(D) = 4$.
 As in the proof of (\ref{index2}), we assume that
$D$ is division   as in  (\ref{disc-2dim} or \ref{disc-2dim-second} ) and not of the type 
    (\ref{disc-2dim}(iii) or \ref{disc-2dim-second}(iii)).
Let $\Gamma$ be the maximal $R$-order of $D$  as in (\ref{maximal-order-2dim}) 
and write  
 $h  = \langle a_1, \cdots , a_n\rangle$  with $a_i \in  \Gamma$. 
  
 Let  $\YY_1 \to Spec(R_0 )$ be a sequence of blowups such that  the support of Nrd$(a_i)$ is a union of 
 regular curves with normal crossings.
 Further replacing $\YY$ by a sequence of blow ups (\ref{notiii}),  we assume that for 
  every closed point $P$ of $\YY$, $ D\otimes F_{0P} $ is not of the form (\ref{disc-2dim}(iii) or \ref{disc-2dim-second}(iii)).  
  Once again we have a finite set of closed points $\PP_1$ of $\YY_1$ such that 
  $X(F_{0P}) \neq \emptyset$ for all $P \not\in \PP$.

Let $P \in \PP_1$ be a closed point. 
Suppose that $F\otimes F_{0P}$ is not a field. Then as in the proof of  (\ref{index2}),
$X(F_{0P}) \neq \emptyset$. 

Suppose that $F\otimes F_{0P}$ is a field. 
 If ind$(D \otimes F_{0P} ) \leq 2$, then  by (\ref{index2}), there exists a sequence of blowups $\XX_P$ of Spec$(\OO_P)$ such that 
 for every closed point $Q$ of $\XX_P$,   $X(F_{0Q}) \neq \emptyset$. 
 
 Suppose ind$(D\otimes F_P ) = 4$.      
 
 Suppose  that $D$ is not of type (\ref{disc-2dim}(iv) or \ref{disc-2dim-second}(iv)).
 Then  $D \simeq (u_0, w_0) \otimes (a, b)$ for some units $u, v, \in R_0$ and $a, b \in R_0 $ as in (\ref{blowup-index2}). 
 Then, by the choice, we have $\Gamma = R(u_0, v_0) \otimes R(a, b)$.
  By (\ref{blowup-index2}), $(a, b )  \otimes  F_{0P} \simeq (a_P, b_P)$ for some $a_P, b_P \in \OO_P$ as in (\ref{blowup-index2}) 
  and $R_0(a, b) \subset \hat{\OO}_P(a_P, b_P)$.  In particular $\Gamma \subset \hat{\OO}_P(u_0, v_0) \otimes \hat{\OO}_P(a_P, b_P)$.
  Since $a_i \in \Gamma$  and $D \otimes F_{0P}$ is not of type (\ref{disc-2dim}(iii)),  by  (\ref{local-results-1}),
 $X(F_{0P}) \neq \emptyset$.

  Suppose  that $D$ is   of type (\ref{disc-2dim}(iv) or  \ref{disc-2dim-second}(iv)).
  Then $D \simeq (u_0, u_1\pi_0) \otimes (v_0, v_1\delta_0)$ for some units $u_i, v_i  \in R_0$.  
  
 Suppose $P$ is a nodal point of $\YY_1$. Then,  by (\ref{blowuppoint-index4}),   there exists an isomorphism 
 $\phi_P : D \otimes F_{0P} \to (u_0', w_1'\pi_P) \otimes (v_0', v_1'\delta_P) $   and $\theta_P \in \hat{\OO}_P(u_0', w_1'\pi_P) \otimes \hat{\OO}_P(v_0', v_1'\delta_P)$
  such that 
 $\phi_P (  R_0(u_0, u_1\pi_0) \otimes R_0(v_0, v_1\pi_0)) \subset  
 \hat{\OO}_P(u_0', w_1'\pi_P) \otimes \hat{\OO}_P(v_0', v_1'\delta_P) $
 and $int(\theta_P) \sigma = \phi_P^{-1}\sigma'\phi_P$, where $\sigma'$ is the product of the canonical involutions on the right hand side. 
 Let $h'$ be the hermitian form on   $( (u_0', w_1'\pi_P) \otimes (v_0', v_1'\delta_P)), \phi_P \sigma \phi_P^{-1})$   which is the image of $h$ under $\phi_P$. 
 Since $h = \langle a_1, \cdots, a_n\rangle$, we have $h_1 = \langle\phi_P(a_1), \cdots , \phi_P(a_n)\rangle$. 
 Let $h' = \theta_P h_1$. Then $h'$ is an hermitian form with  respect $\sigma'$. 
 Let $X'$ be the projective homogeneous variety under $G( (u_0', w_1'\pi_P) \otimes (v_0', v_1'\delta_P),  \sigma', h')$ associated to $X$. 
 Then $X(F_{0P}) \neq \emptyset$ if and only if $X'(F_{0P}) \neq \emptyset$. 
 Since $\phi_P(a_i) , \theta_P \in  \hat{\OO}_P(u_0', w_1'\pi_P) \otimes \hat{\OO}_P(v_0', v_1'\delta_P)$,
 by  (\ref{local-results-1}), $X'(F_{0P}) \neq \emptyset$ and hence $X(F_{0P}) \neq \emptyset$.  
 
 If $P$ is a non-nodal point, then using (\ref{curve-point-index4}), we get $X(F_{0P}) \neq \emptyset$ as above. 
 \end{proof}

 \begin{prop}
 \label{index8} 
   Let $k$, $F_0$ and $F$ be as above.  
    Suppose that for finite extension $\ell/k$, every element in 
  ${}_2Br(\ell)$ has index at most 2. 
   Let $D \in {}_2Br(F)$ be  a central simple  algebra over $F$ with 
 a $F/F_0$-involution $\sigma$  and $h$ an hermitian form over $(D, \sigma)$.
  Let $X$ be a projective homogeneous variety under $G(D, \sigma, h)$ over $F_0$.
 If $X(F_{0\nu}) \neq \emptyset$  for all  divisorial discrete valuations $\nu$ of $F_0$, 
 then  there exists a sequence of blowups $\YY \to Spec(R_0)$ such that for every  closed point $P$ of $\YY$,  
 $X(F_{0P}) \neq \emptyset$. 
 \end{prop}
 
 \begin{proof}  By (\ref{index4}), we assume that   ind($D) > 4$.  
 As in the proof of (\ref{index2}), we assume that $D$ is unramified on $R_0$ except possibly at 
 $(\pi_0)$ and  $(\delta_0)$.   Then,  by (\ref{deg8}, \ref{disc-2dim-second8}),  ind$(D) = 8$ and 
$D\simeq (u_0, u_1) \otimes (v_0, v_1\pi_0) \otimes (w_0, w_1\delta_0) \otimes F$
for some units $u_, v_i, w_i  \in  R_0$. 
 Without loss of generality we assume that 
 $\sigma$ is the tensor product of canonical involutions on
$ (u_0, u_1)$, $(v_0, v_1\pi)$ and $(w_0, w_1\delta)$ and $\tau$. 
Let $\Gamma = R(u_0, u_1) \otimes R(v_0, v_1\pi) \otimes R(w_0, w_1\delta)$. Then 
 $\Gamma$ is a maximal $R$-order of $D$.    We have $h = \langle a_1, \cdots, a_n\rangle$ for some $a_i \in \Gamma$. 
 
 Let $\XX_0 \to Spec(R_0)$ be a sequence of blowups such that  
 the integral closure of $\YY_0$ in $F$ is regular and the support of $Nrd(a_i)$   and ram$_\YY(D)$ is a 
 union of regular curves with normal crossings. 
   By (\ref{finite}), there exists a  finite subset $\PP$ of $Y_0$ such that 
$X(F_{0P} ) \neq \emptyset $ for all $P \not \in \PP_0$. 

Let $P \in \PP_0$.  
Suppose that $F\otimes F_{0P}$ is not a field. Then as in the proof of  (\ref{index2}),
$X(F_{0P}) \neq \emptyset$. 

Suppose that $F\otimes F_{0P}$ is a field. 
If ind$(D \otimes F_{0P}) \leq 4$, then by (\ref{index4}),  
there exists a sequence of blowups $\XX_P$ of Spec$(\OO_P)$ such that 
 for every closed point $Q$ of $\XX_P$,   $X(F_{0Q}) \neq \emptyset$. 
 
 Suppose ind$(D \otimes F_{0P}) >  4$.  Then as above we have  ind$(D \otimes F_{0P}) =  8$.
 Arguing as in the proof of (\ref{index4}), we get that    $X(F_{0P}) \neq \emptyset$. 
  \end{proof}

\section{Main theorem}
\label{main}

In this section we prove the main theorems. 
\begin{theorem} \label{main-theorem1}
 Let $K$ be a complete discretely valued field with valuation ring $T$ and residue field $k$.
 Suppose that char$(k) \neq 2$. Let $F $ be the function field of a smooth projective geometrically 
 integral curve over $K$.   
 Let $A \in {}_2Br(F)$ be a central simple algebra over $F$ with an  involution  $\sigma$ of any kind, $F_0 = F^\sigma$
  and $h$ a hermitian form over $(A, \sigma)$. 
 Suppose that ind$(A) \leq 4$.  Let $G = SU(A, \sigma, h)$ if $\sigma$ is first kind or 
 $U(A,  \sigma, h)$ if $\sigma$ is of second kind.  
  Let $X$ be a projective homogeneous variety under $G$ over $F_0$.
 If $X(F_{0\nu}) \neq \emptyset$ for all divisorial  discrete valuations $\nu$ of $F $, then $X(F_0 ) \neq \emptyset$.  
 \end{theorem}
 
\begin{proof}
Since $F$ is the function field of a curve over $K$, $F_0$ is also 
the function field of a curve over $K$. 
Let $\XX_0 \to Spec(R_0)$ be a regular proper model of $F_0$  with the closed fibre $X_0$
a union of regular curves with normal crossings (\cite{abhyankar}).  
Let $\eta \in   X_0$ be a  generic point.  Then $\eta$ gives a divisorial discrete valuation $\nu$ of $F $.
Since $X(F_{0\eta}) \neq \emptyset$, by (\cite[5.8]{HHK-torsors}), there exists a nonempty open set $U_\eta$ of the closure of
 $\eta$ in $X_0$ such that 
$X(F_{0U_\eta}) \neq \emptyset$. By shrinking $U_\eta$, we assume that $U_\eta$ does not contain any singular points of $X_0$. 

Let $\PP = X_0 \setminus \cup_\eta U_\eta$.  Then $\PP$ is a finite set of closed points of $X_0$ containing all the singular points of $X_0$.  Let $P$ be a closed point of $X_0$.  Suppose $P \not\in \PP$. Then $P \in U_\eta$ for some 
$\eta$. Since $F_{0U_\eta} \subset F_{0P}$,  $X(F_{0P}) \neq \emptyset$.

Let $P \in \PP$. Since ind$(A) \leq 4$ and $A \in {}_2Br(F)$, by (\ref{index4}), there exists a sequence 
of blowups $\XX_P$ of Spec$(\OO_P)$ such that 
$X(F_{0Q}) \neq \emptyset$ for all closed points of $\XX_P$. Thus replacing $\XX$ by 
these finitely many sequences of blowups at all $P \in \PP$, we assume that 
$X(F_{0Q}) \neq \emptyset$ for all closed points $Q$ of $X_0$.  
 Since for any generic point  $\eta$ of $X_0$,   $X(F_{0\eta}) \neq \emptyset$, 
we have $X(F_{0x}) \neq \emptyset$ for all   points $x  \in X_0$. 
Since $G$ is a connected  rational group (\cite[Lemma 5]{chernusov}), by (\cite[Theorem 3.7]{HHK-quadratic}), we have $X(F) \neq \emptyset$. 
\end{proof}

\begin{theorem}
 \label{main-theorem2}
  Let $K$ be a complete discretely valued field with valuation ring $T$ and residue field $k$.
 Suppose that char$(k) \neq 2$. Let $F $ be the function field of a smooth projective geometrically 
 integral curve over $K$.   
 Let $A$ be a central simple algebra over $F$ with an  $\sigma$ of any kind, $F_0 = F^\sigma$
  and $h$ a hermitian form over $(A, \sigma)$. 
 Suppose that  for  every finite extension $\ell/k$, every element in ${}_2Br(\ell)$ has index at most 2. 
 Let $G = SU(A, \sigma, h)$ if $\sigma$ is first kind or 
 $U(A,  \sigma, h)$ if $\sigma$ is of second kind. 
 Let $X$ be a projective homogeneous variety under $G$ over $F_0$.
 If $X(F_{0\nu}) \neq \emptyset$ for all divisorial  discrete valuations $\nu$ of $F_0 $, then $X(F_0 ) \neq \emptyset$.  
 \end{theorem}
 
\begin{proof}   Using (\ref{index8}), the   proof  is similar to the proof of (\ref{main-theorem1}).
\end{proof}

\begin{cor}
\label{cor1}
 Let $K$ be a complete discretely valued field with valuation ring $T$ and residue field $k$.
 Suppose that $k$ is a global field, local field or a $C_2$-field  with char$(k) \neq 2$. 
 Let $F $ be the function field of a smooth projective geometrically 
 integral curve over $K$.   
 Let $A \in {}_2Br(F)$ be a central simple algebra over $F$ with an   involution $\sigma$ of any kind and
 $h$ a hermitian form over $(A, \sigma)$. 
 Let  $F_0 = F^\sigma$, $G = SU(A, \sigma, h)$ if $\sigma$ is first kind and $G = 
 U(A,  \sigma, h)$ if $\sigma$ is of second kind. 
 Let $X$ be a projective homogeneous variety under $G$ over $F_0$.
 If $X(F_{0\nu}) \neq \emptyset$ for all divisorial  discrete valuations $\nu$ of $F_0 $, then $X(F_0 ) \neq \emptyset$.  

 \end{cor}
 
 \begin{proof} Suppose $k$ is a global field or a local field or a $C_2$-field, then  any finite extension $\ell$ of $k$ is
 also same type and hence every element in 
 $_2Br(k)$ is of index at most 2 (\cite[Chapter 10, 2.3(vi)]{scharlau}, \cite[Chapter 10, 2.2(i)]{scharlau}, \cite[Theorem 4.8]{lam}).  Hence the corollary follows from (\ref{main-theorem2}).  
 \end{proof}

\begin{cor}
\label{springer}
 Let $K$ be a complete discretely valued field with valuation ring $T$ and residue field $k$.
 Suppose that  $k$ is a local field   with  the characteristic of the  residue field   not equal  2. 
 Let $F $ be the function field of a smooth projective geometrically 
 integral curve over $K$.   
 Let $A \in {}_2Br(F)$ be a central simple algebra over $F$ with an   involution $\sigma$   and
 $h$ a hermitian form over $(A, \sigma)$ and $F =  F^\sigma$.   
If $h \otimes L$ is isotropic for some odd degree extension $L/F_0$, then $h$ is isotropic. 
 \end{cor}

\begin{proof}   
Let $\nu$ be a  divisorial discrete valuation of $F_0$. 
Then the residue field $\kappa(\nu)$ is either a finite extension of $K$ or a function field of a curve over 
a finite extension of $k$ (\cite[Theorem 8.1]{parimalaHassequadraticfunctionfields}). 

Let $\kappa'$ be an extension of $\kappa(\nu)$ with $[\kappa' : \kappa(\nu)] \leq 2$.  
Let $D$ be a central simple algebra over $\kappa'$ with a $\kappa'/\kappa(\nu')$-involution $\theta$ and 
$f$ an hermitian form over $(D, \theta)$.
Let $\ell$ be an extension of $\kappa(\nu)$ of  odd degree. Suppose that $f$ is isotropic over 
$\ell$.

Suppose  $\kappa(\nu)$ is  a finite extension of $K$.   Since 
$K$ is a complete discretely valued field with residue field $k$, $\kappa(\nu)$ 
is a compete discretely valued field with residue field a finite extension of $k$. 
Then by  (\cite[Lemma 5.6]{wu} and  \cite[Lemma 5.1]{wu}), we get that 
$f  $ is isotropic. 

 Suppose the residue field $\kappa(\nu)$ is  a function field of a curve over 
a finite extension of $k$. 
Since $k$ is local field, by (\cite[Lemma 5.1]{wu}) and  (\cite[Theorem 4.4]{wu}), 
$ f$ is isotropic.

Suppose that $h\otimes L$ is isotropic for some odd degree extension $L$ of $F_0$.
Then,  as in the proof of (\cite[Theorem 5.8]{wu}) and using (\ref{cor1}), we get that $h$ is isotropic.

\end{proof}

\printbibliography

@incollection {abhyankar,
    AUTHOR = {Abhyankar, Shreeram Shankar},
     TITLE = {Resolution of singularities of algebraic surfaces},
 BOOKTITLE = {Algebraic {G}eometry ({I}nternat. {C}olloq., {T}ata {I}nst.
              {F}und. {R}es., {B}ombay, 1968)},
     PAGES = {1--11},
 PUBLISHER = {Oxford Univ. Press, London},
      YEAR = {1969},
   MRCLASS = {14.18},
  MRNUMBER = {0257080},
MRREVIEWER = {Joseph Lipman},
}

@book {albert,
    AUTHOR = {Albert, A. Adrian},
     TITLE = {Structure of algebras},
    SERIES = {American Mathematical Society Colloquium Publications, Vol.
              XXIV},
      NOTE = {Revised printing},
 PUBLISHER = {American Mathematical Society, Providence, R.I.},
      YEAR = {1961},
     PAGES = {xi+210},
   MRCLASS = {16.50},
  MRNUMBER = {0123587},
}

@article {auslander,
    AUTHOR = {Auslander, Maurice and Goldman, Oscar},
     TITLE = {The {B}rauer group of a commutative ring},
   JOURNAL = {Trans. Amer. Math. Soc.},
  FJOURNAL = {Transactions of the American Mathematical Society},
    VOLUME = {97},
      YEAR = {1960},
     PAGES = {367--409},
      ISSN = {0002-9947},
   MRCLASS = {18.00 (16.00)},
  MRNUMBER = {121392},
MRREVIEWER = {T. Nakayama},
       DOI = {10.2307/1993378},
       URL = {https://doi.org/10.2307/1993378},
}

@article {chernusov,
    AUTHOR = {Chernousov, Vladimir I. and Platonov, Vladimir P.},
     TITLE = {The rationality problem for semisimple group varieties},
   JOURNAL = {J. Reine Angew. Math.},
  FJOURNAL = {Journal f\"{u}r die Reine und Angewandte Mathematik. [Crelle's
              Journal]},
    VOLUME = {504},
      YEAR = {1998},
     PAGES = {1--28},
      ISSN = {0075-4102},
   MRCLASS = {14L35 (14M20 20G15)},
  MRNUMBER = {1656830},
MRREVIEWER = {Andy R. Magid},
       DOI = {10.1515/crll.1998.108},
       URL = {https://doi.org/10.1515/crll.1998.108},
}

@article {HHK-quadratic,
    AUTHOR = {Harbater, David and Hartmann, Julia and Krashen, Daniel},
     TITLE = {Applications of patching to quadratic forms and central simple
              algebras},
   JOURNAL = {Invent. Math.},
  FJOURNAL = {Inventiones Mathematicae},
    VOLUME = {178},
      YEAR = {2009},
    NUMBER = {2},
     PAGES = {231--263},
      ISSN = {0020-9910},
   MRCLASS = {11E04 (16K20)},
  MRNUMBER = {2545681},
MRREVIEWER = {Mohammad G. Mahmoudi},
       DOI = {10.1007/s00222-009-0195-5},
       URL = {https://doi.org/10.1007/s00222-009-0195-5},
}

@article {HHK-torsors,
    AUTHOR = {Harbater, David and Hartmann, Julia and Krashen, Daniel},
     TITLE = {Local-global principles for torsors over arithmetic curves},
   JOURNAL = {Amer. J. Math.},
  FJOURNAL = {American Journal of Mathematics},
    VOLUME = {137},
      YEAR = {2015},
    NUMBER = {6},
     PAGES = {1559--1612},
      ISSN = {0002-9327},
   MRCLASS = {14G05 (11E72 11G05 14H05 16K50)},
  MRNUMBER = {3432268},
MRREVIEWER = {Skip Garibaldi},
       DOI = {10.1353/ajm.2015.0039},
       URL = {https://doi.org/10.1353/ajm.2015.0039},
}

@article {HHK-refinements,
    AUTHOR = {Harbater, David and Hartmann, Julia and Krashen, Daniel},
     TITLE = {Refinements to patching and applications to field invariants},
   JOURNAL = {Int. Math. Res. Not. IMRN},
  FJOURNAL = {International Mathematics Research Notices. IMRN},
      YEAR = {2015},
    NUMBER = {20},
     PAGES = {10399--10450},
      ISSN = {1073-7928},
   MRCLASS = {14F22 (11E04 11E81 12J10 16K50)},
  MRNUMBER = {3455871},
MRREVIEWER = {Timothy J. Ford},
       DOI = {10.1093/imrn/rnu278},
       URL = {https://doi.org/10.1093/imrn/rnu278},
}

@book {knus-hermitian,
    AUTHOR = {Knus, Max-Albert},
     TITLE = {Quadratic and {H}ermitian forms over rings},
    SERIES = {Grundlehren der mathematischen Wissenschaften [Fundamental
              Principles of Mathematical Sciences]},
    VOLUME = {294},
      NOTE = {With a foreword by I. Bertuccioni},
 PUBLISHER = {Springer-Verlag, Berlin},
      YEAR = {1991},
     PAGES = {xii+524},
      ISBN = {3-540-52117-8},
   MRCLASS = {11Exx (11E39 11E81 16E20 19Gxx)},
  MRNUMBER = {1096299},
MRREVIEWER = {Rudolf Scharlau},
       DOI = {10.1007/978-3-642-75401-2},
       URL = {https://doi.org/10.1007/978-3-642-75401-2},
}

@book {knus-involutions,
    AUTHOR = {Knus, Max-Albert and Merkurjev, Alexander and Rost, Markus and
              Tignol, Jean-Pierre},
     TITLE = {The book of involutions},
    SERIES = {American Mathematical Society Colloquium Publications},
    VOLUME = {44},
      NOTE = {With a preface in French by J. Tits},
 PUBLISHER = {American Mathematical Society, Providence, RI},
      YEAR = {1998},
     PAGES = {xxii+593},
      ISBN = {0-8218-0904-0},
   MRCLASS = {16K20 (11E39 11E57 11E72 11E88 16W10 20G10)},
  MRNUMBER = {1632779},
MRREVIEWER = {A. R. Wadsworth},
       DOI = {10.1090/coll/044},
       URL = {https://doi.org/10.1090/coll/044},
}

@misc{parimala2020localglobal,
      title={Local-Global Principle for Unitary Groups Over Function Fields of p-adic Curves}, 
      author={R. Parimala and V. Suresh},
      year={2020},
      eprint={2004.10357},
      archivePrefix={arXiv},
      primaryClass={math.NT}
}

@article {reddy,
    AUTHOR = {Reddy, B. Surendranath and Suresh, V.},
     TITLE = {Admissibility of groups over function fields of p-adic curves},
   JOURNAL = {Adv. Math.},
  FJOURNAL = {Advances in Mathematics},
    VOLUME = {237},
      YEAR = {2013},
     PAGES = {316--330},
      ISSN = {0001-8708},
   MRCLASS = {20D20 (14H25 16Kxx)},
  MRNUMBER = {3028580},
MRREVIEWER = {Zinovy Reichstein},
       DOI = {10.1016/j.aim.2012.12.017},
       URL = {https://doi.org/10.1016/j.aim.2012.12.017},
}

@book {reiner,
    AUTHOR = {Reiner, I.},
     TITLE = {Maximal orders},
    SERIES = {London Mathematical Society Monographs. New Series},
    VOLUME = {28},
      NOTE = {Corrected reprint of the 1975 original,
              With a foreword by M. J. Taylor},
 PUBLISHER = {The Clarendon Press, Oxford University Press, Oxford},
      YEAR = {2003},
     PAGES = {xiv+395},
      ISBN = {0-19-852673-3},
   MRCLASS = {16H05 (11R54 16K20)},
  MRNUMBER = {1972204},
}

@article {saltman-cyclic,
    AUTHOR = {Saltman, David J.},
     TITLE = {Cyclic algebras over {$p$}-adic curves},
   JOURNAL = {J. Algebra},
  FJOURNAL = {Journal of Algebra},
    VOLUME = {314},
      YEAR = {2007},
    NUMBER = {2},
     PAGES = {817--843},
      ISSN = {0021-8693},
   MRCLASS = {16K20 (11S15 16K50)},
  MRNUMBER = {2344586},
MRREVIEWER = {Jan van Geel},
       DOI = {10.1016/j.jalgebra.2007.03.003},
       URL = {https://doi.org/10.1016/j.jalgebra.2007.03.003},
}

@article {saltman-divison,
    AUTHOR = {Saltman, David J.},
     TITLE = {Division algebras over {$p$}-adic curves},
   JOURNAL = {J. Ramanujan Math. Soc.},
  FJOURNAL = {Journal of the Ramanujan Mathematical Society},
    VOLUME = {12},
      YEAR = {1997},
    NUMBER = {1},
     PAGES = {25--47},
      ISSN = {0970-1249},
   MRCLASS = {16H05 (12E15 13A20 16K20)},
  MRNUMBER = {1462850},
MRREVIEWER = {Timothy J. Ford},
}

@book {scharlau,
    AUTHOR = {Scharlau, Winfried},
     TITLE = {Quadratic and {H}ermitian forms},
    SERIES = {Grundlehren der mathematischen Wissenschaften [Fundamental
              Principles of Mathematical Sciences]},
    VOLUME = {270},
 PUBLISHER = {Springer-Verlag, Berlin},
      YEAR = {1985},
     PAGES = {x+421},
      ISBN = {3-540-13724-6},
   MRCLASS = {11Exx (11-02 12D15 15A63 16A16 16A28)},
  MRNUMBER = {770063},
MRREVIEWER = {R. Ware},
       DOI = {10.1007/978-3-642-69971-9},
       URL = {https://doi.org/10.1007/978-3-642-69971-9},
}

@article {wu,
    AUTHOR = {Wu, Zhengyao},
     TITLE = {Hasse principle for hermitian spaces over semi-global fields},
   JOURNAL = {J. Algebra},
  FJOURNAL = {Journal of Algebra},
    VOLUME = {458},
      YEAR = {2016},
     PAGES = {171--196},
      ISSN = {0021-8693},
   MRCLASS = {11E39 (11E72 14G05 20G35)},
  MRNUMBER = {3500773},
MRREVIEWER = {Demba Barry},
       DOI = {10.1016/j.jalgebra.2016.02.027},
       URL = {https://doi.org/10.1016/j.jalgebra.2016.02.027},
}

@article {parimala-uinvariant,
    AUTHOR = {Parimala, R. and Suresh, V.},
     TITLE = {Period-index and {$u$}-invariant questions for function fields
              over complete discretely valued fields},
   JOURNAL = {Invent. Math.},
  FJOURNAL = {Inventiones Mathematicae},
    VOLUME = {197},
      YEAR = {2014},
    NUMBER = {1},
     PAGES = {215--235},
      ISSN = {0020-9910},
   MRCLASS = {16K50 (11E04 11E08 11R58 12G05 12J10 16K20)},
  MRNUMBER = {3219517},
MRREVIEWER = {Detlev W. Hoffmann},
       DOI = {10.1007/s00222-013-0483-y},
       URL = {https://doi.org/10.1007/s00222-013-0483-y},
}

@article {merkurjevflagvar,
    AUTHOR = {Merkurjev, A. S. and Panin, I. A. and Wadsworth, A. R.},
     TITLE = {Index reduction formulas for twisted flag varieties. {I}},
   JOURNAL = {$K$-Theory},
  FJOURNAL = {$K$-Theory. An Interdisciplinary Journal for the Development,
              Application, and Influence of $K$-Theory in the Mathematical
              Sciences},
    VOLUME = {10},
      YEAR = {1996},
    NUMBER = {6},
     PAGES = {517--596},
      ISSN = {0920-3036},
   MRCLASS = {16K20 (14L30 19E08 20G05)},
  MRNUMBER = {1415325},
MRREVIEWER = {Jean-Pierre Tignol},
       DOI = {10.1007/BF00537543},
       URL = {https://doi.org/10.1007/BF00537543},
}

@article {merkurjevflagvar2,
    AUTHOR = {Merkurjev, A. S. and Panin, I. A. and Wadsworth, A. R.},
     TITLE = {Index reduction formulas for twisted flag varieties. {II}},
   JOURNAL = {$K$-Theory},
  FJOURNAL = {$K$-Theory. An Interdisciplinary Journal for the Development,
              Application, and Influence of $K$-Theory in the Mathematical
              Sciences},
    VOLUME = {14},
      YEAR = {1998},
    NUMBER = {2},
     PAGES = {101--196},
      ISSN = {0920-3036},
   MRCLASS = {16K20 (14L30 19E08 20G05)},
  MRNUMBER = {1628279},
MRREVIEWER = {Jean-Pierre Tignol},
       DOI = {10.1023/A:1007793218556},
       URL = {https://doi.org/10.1023/A:1007793218556},
}

@article {karpenko,
    AUTHOR = {Karpenko, N. A.},
     TITLE = {Cohomology of relative cellular spaces and of isotropic flag
              varieties},
   JOURNAL = {Algebra i Analiz},
  FJOURNAL = {Rossi\u{\i}skaya Akademiya Nauk. Algebra i Analiz},
    VOLUME = {12},
      YEAR = {2000},
    NUMBER = {1},
     PAGES = {3--69},
      ISSN = {0234-0852},
   MRCLASS = {14M15 (57T15)},
  MRNUMBER = {1758562},
MRREVIEWER = {I. Dolgachev},
}

@article {patchingLGPpadic,
    AUTHOR = {Colliot-Th\'{e}l\`ene, Jean-Louis and Parimala, Raman and Suresh,
              Venapally},
     TITLE = {Patching and local-global principles for homogeneous spaces
              over function fields of {$p$}-adic curves},
   JOURNAL = {Comment. Math. Helv.},
  FJOURNAL = {Commentarii Mathematici Helvetici. A Journal of the Swiss
              Mathematical Society},
    VOLUME = {87},
      YEAR = {2012},
    NUMBER = {4},
     PAGES = {1011--1033},
      ISSN = {0010-2571},
   MRCLASS = {11G99 (11E12 11E72 14G05 14G20 20G35)},
  MRNUMBER = {2984579},
MRREVIEWER = {Jan van Geel},
       DOI = {10.4171/CMH/276},
       URL = {https://doi.org/10.4171/CMH/276},
}

@article {parimalaHassequadraticfunctionfields,
    AUTHOR = {Parimala, R.},
     TITLE = {A {H}asse principle for quadratic forms over function fields},
   JOURNAL = {Bull. Amer. Math. Soc. (N.S.)},
  FJOURNAL = {American Mathematical Society. Bulletin. New Series},
    VOLUME = {51},
      YEAR = {2014},
    NUMBER = {3},
     PAGES = {447--461},
      ISSN = {0273-0979},
   MRCLASS = {11E04 (11-03 11G35 11R58)},
  MRNUMBER = {3196794},
MRREVIEWER = {Pete L. Clark},
       URL = {https://doi.org/10.1090/S0273-0979-2014-01443-0},
}

@book {lam,
    AUTHOR = {Lam, T. Y.},
     TITLE = {Introduction to quadratic forms over fields},
    SERIES = {Graduate Studies in Mathematics},
    VOLUME = {67},
 PUBLISHER = {American Mathematical Society, Providence, RI},
      YEAR = {2005},
     PAGES = {xxii+550},
      ISBN = {0-8218-1095-2},
   MRCLASS = {11Exx},
  MRNUMBER = {2104929},
MRREVIEWER = {K. Szymiczek},
       DOI = {10.1090/gsm/067},
       URL = {https://doi.org/10.1090/gsm/067},
}

@misc{colliotthélène2021localglobal,
      title={Local-global principles for constant reductive groups over semi-global fields}, 
      author={Jean-Louis Colliot-Thélène and David Harbater and Julia Hartmann and Daniel Krashen and R. Parimala and V. Suresh},
      year={2021},
      eprint={2108.12349},
      archivePrefix={arXiv},
      primaryClass={math.AG}
}

@article{tori,
   title={Local-global principles for tori over arithmetic curves},
   ISSN={2214-2584},
   url={http://dx.doi.org/10.14231/AG-2020-022},
   DOI={10.14231/ag-2020-022},
   journal={Algebraic Geometry},
   publisher={Foundation Compositio Mathematica},
   author={Colliot-Thélène, Jean-Louis and Harbater, David and Hartmann, Julia and Krashen, Daniel and Parimala, Raman and Suresh, Venapally},
   year={2020},
   month={Sep},
   pages={607–633} }

\end{document}